\documentclass{amsart}
\usepackage{amssymb}
\usepackage[all]{xy}

\catcode`\@=11

\newlength{\tabwidth}
\newlength{\tabheight}
\newcommand{\tabstyle}{\textstyle}
\newlength{\tabrulel}
\newlength{\tabruler}
\newlength{\tabrulet}
\newlength{\tabruleb}
\newlength{\tabwidthx}
\newlength{\tabheightx}

\def\gentabbox#1#2#3#4#5#6#7{\vbox to \tabheight{%
  \setlength{\tabrulel}{#3}\setlength{\tabruler}{#4}%
  \setlength{\tabrulet}{#5}\setlength{\tabruleb}{#6}%
  \setlength{\tabwidthx}{#1\tabwidth}\addtolength{\tabwidthx}{0.5\tabrulel}%
    \addtolength{\tabwidthx}{0.5\tabruler}%
  \setlength{\tabheightx}{#2\tabheight}\addtolength{\tabheightx}{-\tabheight}%
  \hbox to #1\tabwidth{%
    \hspace{-0.5\tabrulel}\rule{\tabrulel}{#2\tabheight}\hspace{-\tabrulel}%
    \vbox to #2\tabheight{\hsize=\tabwidthx%
      \vspace{-0.5\tabrulet}\hrule width\tabwidthx height\tabrulet%
      \vspace{-0.5\tabrulet}\vss%
      \hbox to \tabwidthx{\hss#7\hss}%
        \vss\vspace{-0.5\tabruleb}%
      \hrule width\tabwidthx height\tabruleb\vspace{-0.5\tabruleb}}%
    \hspace{-\tabruler}\rule{\tabruler}{#2\tabheight}\hspace{-0.5\tabruler}}%
  \vspace{-\tabheightx}}}
\def\genblankbox#1#2{\vbox to \tabheight{\vfil\hbox to #1\tabwidth{\hfil}}}
\def\tabbox#1#2#3{\gentabbox{#1}{#2}{0.4pt}{0.4pt}{0.4pt}{0.4pt}{#3}}
\def\boldtabbox#1#2#3{\gentabbox{#1}{#2}{1.2pt}{1.2pt}{1.2pt}{1.2pt}{#3}}

\catcode`\:=13 \catcode`\.=13 \catcode`\;=13 
\catcode`\>=13 \catcode`\^=13
\def:#1\\{\hbox{$\tabstyle #1$}}
\def.#1{\tabbox{1}{1}{$\tabstyle #1$}}
\def>#1{\tabbox{2}{1}{$\tabstyle #1$}}
\def^#1{\tabbox{1}{2}{$\tabstyle #1$}}
\def;{\genblankbox{1}{1}\relax}
\catcode`\:=12 \catcode`\.=12 \catcode`\;=12 
\catcode`\>=12 \catcode`\^=7

\newenvironment{tableau}{\bgroup\catcode`\:=13 \catcode`\.=13
  \catcode`\;=13 \catcode`\>=13 \catcode`\^=13 
  \def\b##1##2##3{\boldtabbox{##1}{##2}{\vbox{##3}}}%
  \def\n##1##2##3{\tabbox{##1}{##2}{\vbox{##3}}}%
  \def\c{\tabbox{1}{1}{{}}}
  \vcenter\bgroup\offinterlineskip}{\egroup\egroup}

\catcode`\@=11

\newcommand{\rc}{\gentabbox{1}{1}{0.4pt}{2pt}{0.4pt}{0.4pt}{{}}}
\newcommand{\lc}{\gentabbox{1}{1}{2pt}{0.4pt}{0.4pt}{0.4pt}{{}}}

\newcommand{\rce}[1]{\gentabbox{1}{1}{0.4pt}{2pt}{0.4pt}{0.4pt}{{$\tabstyle #1$}}}
\newcommand{\lce}[1]{\gentabbox{1}{1}{2pt}{0.4pt}{0.4pt}{0.4pt}{{$\tabstyle #1$}}}
\newcommand{\brce}[1]{\gentabbox{1}{1}{1pt}{2pt}{1pt}{1pt}{{$\tabstyle #1$}}}

\newcommand{\bce}[1]{\boldtabbox{1}{1}{{$\tabstyle #1$}}}
\newcommand{\row}[1]{\hbox{$\tabstyle #1$}}
\newcommand{\caprow}[1]{\hbox{$\scriptscriptstyle #1$}}
\newcommand{\q}{\genblankbox{1}{1}\relax}
\newcommand{\sq}{\genblankbox{0.5}{1}\relax}
\newcommand{\cp}{\gentabbox{1}{1}{0.4pt}{0.4pt}{0.4pt}{0.4pt}{{$\tabstyle +$}}}
\newcommand{\cm}{\gentabbox{1}{1}{0.4pt}{0.4pt}{0.4pt}{0.4pt}{{$\tabstyle -$}}}
\renewcommand{\tabstyle}{\scriptscriptstyle}

\newtheorem{thm}{Theorem}[section]
\newtheorem{lem}[thm]{Lemma}
\newtheorem{prop}[thm]{Proposition}
\newtheorem{cor}[thm]{Corollary}
\newtheorem{conj}[thm]{Conjecture}

\theoremstyle{definition}
\newtheorem{defn}[thm]{Definition}

\theoremstyle{remark}
\newtheorem{rmk}[thm]{Remark}
\newtheorem{exam}[thm]{Example}

\numberwithin{equation}{section}

\newcommand{\C}{\mathbb{C}}
\newcommand{\F}{\mathbb{F}}

\newcommand{\cC}{\mathcal{C}}

\newcommand{\cN}{\mathcal{N}}
\newcommand{\cO}{\mathcal{O}}
\newcommand{\cP}{\mathcal{P}}
\newcommand{\cQ}{\mathcal{Q}}

\newcommand{\gl}{\mathfrak{gl}}
\newcommand{\fn}{\mathfrak{n}}

\newcommand{\SP}{\mathcal{SP}}
\newcommand{\SQ}{\mathcal{SQ}}

\newcommand{\bt}{{\mathbf{t}}}
\newcommand{\Ib}{{\overline{I}}}
\newcommand{\isomto}{\overset{\sim}{\rightarrow}}
\newcommand{\git}{\setminus\negthickspace\negthinspace\setminus}

\newcommand{\half}{{\textstyle\frac{1}{2}}}

\DeclareMathOperator{\End}{End}
\DeclareMathOperator{\Hom}{Hom}

\DeclareMathOperator{\Spec}{Spec}

\DeclareMathOperator{\Gr}{Gr}

\DeclareMathOperator{\codim}{codim}
\DeclareMathOperator{\im}{im}

\title[Normality of orbit closures]
{Normality of orbit closures in the enhanced nilpotent cone}
\author{Pramod N. Achar}
\address{Department of Mathematics\\
Louisiana State University\\
Baton Rouge, LA 70803}
\email{pramod@math.lsu.edu}
\author{Anthony Henderson}
\address{School of Mathematics and Statistics\\
University of Sydney NSW 2006\\
Australia}
\email{anthony.henderson@sydney.edu.au}
\author{Benjamin F. Jones}
\address{Department of Mathematics\\
University of Georgia\\
Athens, GA 30602-7403}
\email{bjones@math.uga.edu}

\thanks{The first author's research was supported by Louisiana
Board of Regents grant NSF(2008)-LINK-35 and by National Security
Agency grant H98230-09-1-0024. The second
author's research was supported by
Australian Research Council grant DP0985184. The third author's research was supported in part by NSF VIGRE grant DMS-0738586.}

\begin{document}

\begin{abstract}
We continue the study of the closures of $GL(V)$-orbits in the enhanced nilpotent cone $V\times\cN$ begun by the first two authors. We prove that each closure is an invariant-theoretic quotient of a suitably-defined enhanced quiver variety. We conjecture, and prove in special cases, that these enhanced quiver varieties are normal complete intersections, implying that the enhanced nilpotent orbit closures are also normal.
\end{abstract}

\maketitle

\section{Introduction}

The geometry of nilpotent orbits in complex semisimple Lie algebras is a topic of central importance in numerous branches of representation theory.  A fundamental question on this topic is: Are the closures of nilpotent orbits normal varieties?  This question was answered in the affirmative for nilpotent orbits in type $A$ by Kraft--Procesi~\cite{kp} in 1979.  In other types, the answer turns out to be ``not always'': an explicit determination of the nilpotent orbits with normal closures was carried out in types $B$ and $C$ by Kraft--Procesi~\cite{kp2}, and in types $G_2$, $F_4$, and $E_6$ by Kraft~\cite{kra}, Broer~\cite{bro}, and Sommers~\cite{som1}, respectively.  The case of type $D$ was partially resolved by Kraft--Procesi~\cite{kp2} and completed by Sommers~\cite{som2}.  A complete answer is not yet known in types $E_7$ and $E_8$.  

The present paper is concerned with the variety $V \times \cN$, where $V$ is a finite-dimensional complex vector space, and $\cN$ is the variety of nilpotent elements in $\End(V)$.  This variety, known as the \emph{enhanced nilpotent cone}, was studied by the first two authors in~\cite{ah}.  It is closely related to Kato's exotic nilpotent cone~\cite{kato:exotic, kato:deformations} and to the work of Travkin~\cite{travkin} together with Finkelberg and Ginzburg~\cite{fgt} on mirabolic character sheaves.  The geometry of $GL(V)$-orbits on $V \times \cN$ resembles that of ordinary type-$A$ nilpotent orbits in some ways (e.g., the only equivariant local systems are trivial), but is reminiscent of types $B$ and $C$ in others (e.g., the orbits are parametrized by bipartitions and the local intersection cohomology of orbit closures is described by type-$B$/$C$ combinatorics~\cite{ah}). The upshot of this paper is that, as regards normality of orbit closures, the enhanced nilpotent cone is analogous to the type-$A$ nilpotent cone. That is, our results contribute to proving the following generalization of~\cite{kp}.

\begin{conj}\label{conj:enhnorm}
The closure of each $GL(V)$-orbit in $V \times \cN$ is normal.
\end{conj}

In this paper, we prove a series of implications, summarized in Figure~\ref{fig:conj}, that reduce Conjecture~\ref{conj:enhnorm} to a combinatorial statement, Conjecture~\ref{conj:strata}. The combinatorics is more complicated than in the unenhanced case studied in~\cite{kp}, and at present we can prove Conjecture~\ref{conj:strata} only for a restricted class of enhanced nilpotent orbits; we have also verified it by computer for orbits in low dimensions. The cases of Conjecture~\ref{conj:enhnorm} which are proved in this paper are listed in Corollary~\ref{cor:main}.  


\begin{figure}
\[
\xymatrix@C=10pt{
*+[F]{\txt{Conjecture~\ref{conj:strata}}} \ar@{=>}[d] &
\hbox to 3.5in{dimension estimate for strata of $\Lambda$\hfill\ } \\
*+[F]{\txt{Conjecture~\ref{conj:dimbound}}} \ar@{=>}[d]^{\txt{Theorem~\ref{thm:nonsingrel}}} &
\hbox to 3.5in{dimension estimate for singular locus of $\Lambda$\hfill} \\
*+[F]{\txt{Conjecture~\ref{conj:lambdanormal}}} \ar@{=>}[d]^{\txt{Theorem~\ref{thm:normrel}}} & 
\hbox to 3.5in{normality of $\Lambda$\hfill} \\
*+[F]{\txt{Conjecture~\ref{conj:enhnorm}}} \ar@{=>}[d] &
\hbox to 3.5in{normality of enhanced nilpotent orbits\hfill} \\
*+[F]{\txt{Theorem~\ref{thm:r1}}} &
\hbox to 3.5in{regularity in codimension~$1$ for enhanced nilpotent orbits\hfill}}
\]
\caption{}\label{fig:conj}
\end{figure}

The main tool in our argument is a new class of spaces called \emph{enhanced quiver varieties}.  These varieties, whose definition (see Section~\ref{sect:eqv}) is inspired by the methods of Kraft--Procesi~\cite{kp}, seem to be interesting in their own right.
In Theorem~\ref{thm:gitquot}, we exhibit the closure of a $GL(V)$-orbit in $V \times \cN$ as an invariant-theoretic quotient of an enhanced quiver variety. So as in \cite{kp}, proving the normality of the enhanced quiver varieties would suffice to prove Conjecture~\ref{conj:enhnorm}. 

\begin{exam} \label{exam:intro}
Here is one example to give the flavour of the general definition. Suppose that $\dim V=4$. The closure of the subregular orbit in $\cN$ is 
\[ \{x\in\cN\,|\,x^3=0\}, \]
which in~\cite{kp} is described as an invariant-theoretic quotient of the following quiver variety: the variety of quadruples $(A_1,B_1,A_2,B_2)$ of linear maps
\[
\xymatrix{
\C^1 \ar@/^/[r]^{A_1} &
\C^2 \ar@/^/[l]^{B_{1}} \ar@/^/[r]^{A_2} &
V \ar@/^/[l]^{B_{2}}
}
\]
satisying the equations $B_1 A_1=0$ and $B_2 A_2=A_1 B_1$. In the enhanced setting, one of the orbit closures in $V\times\cN$ is 
\[ \{(v,x)\in V\times\cN\,|\,x^3=0,\,x^2 v=0\}. \]
We will describe this as an invariant-theoretic quotient of the following enhanced quiver variety: the variety of sextuples $(u,v,A_1,B_1,A_2,B_2)$ where $(A_1,B_1,A_2,B_2)$ is as above, $u\in\C^1$, $v\in V$, and $A_1 u=B_2 v$.
\end{exam}

We begin in Section~\ref{sect:review} by fixing notation and conventions for partitions and related combinatorial objects, and by recalling relevant facts about enhanced nilpotent orbits and related varieties.  Section~\ref{sect:ivt} is devoted to the proof of a preparatory result on quotients of the space of `enhanced nilpotent pairs'.  Enhanced quiver varieties are introduced in Section~\ref{sect:eqv}, which also contains 
the proof that their quotients are isomorphic to the enhanced nilpotent orbit closures.  
The next two sections carry out further study of the geometry of enhanced quiver varieties, and conclude with a proof of their normality in certain cases.  The aforementioned combinatorial conjecture is stated and discussed in Section~\ref{sect:stratdim}.  Finally, in Section~\ref{sect:r1}, which is somewhat independent of the rest of the paper, we prove that all enhanced nilpotent orbit closures are regular in codimension~$1$.  This is, of course, a necessary condition for Conjecture~\ref{conj:enhnorm} to hold; and it is not immediately obvious, because enhanced nilpotent orbits can have orbits of codimension~$1$ in their boundary. 

The results of Section~\ref{sect:r1} hold over an arbitrary algebraically closed field, which raises the possibility that Conjecture~\ref{conj:enhnorm} may also be true in positive characteristic. The method of proof suggested in this paper follows \cite{kp} in assuming that the characteristic is zero, but it is possible that it could be adapted to positive characteristic with the techniques used by Donkin \cite{donkin} in the unenhanced case.  

\textbf{Acknowledgements.} We are grateful to C.~Johnson for sending us an early version of \cite{johnson}, and to D.~Nakano for helpful conversations.

\section{Partitions and Nilpotent Matrices}
\label{sect:review}

In this section, we fix notation related to the combinatorics of partitions and bipartitions, and we review relevant results on nilpotent orbits, nilpotent pairs, and enhanced versions thereof. These results hold over any field, but we use $\C$ for the sake of subsequent sections.

\subsection{Compositions, partitions, bipartitions}

A \emph{composition} is a sequence $\lambda = (\lambda_1,\lambda_2,\ldots)$ of nonnegative integers with finitely many nonzero terms.  The \emph{size} of a composition, denoted $|\lambda|$, is the sum of its terms.  The infinite tail of $0$'s will typically be omitted when writing a composition.

A \emph{partition} is a composition $(\lambda_1,\lambda_2, \ldots)$ with $\lambda_1 \ge \lambda_2 \ge \cdots$.  The \emph{length} $\ell(\lambda)$ of a partition $\lambda$ is the number of nonzero terms.  Partitions are often written with exponents indicating multiplicities: for instance, we may write $3^21^3$ rather than $(3,3,1,1,1)$.  Let
\[
\cP_n = \{ \text{partitions of size $n$} \}.
\]
A \emph{bipartition of size $n$} is simply an ordered pair $(\mu;\nu)$ of partitions with $|\mu| + |\nu| = n$.  We put
\[
\cQ_n = \{ \text{bipartitions of size $n$} \}.
\]
Given a bipartition $(\mu;\nu)$, we can form a partition in two ways: the \emph{sum} $\mu + \nu$ (obtained by termwise addition of sequences) and the \emph{union} $\mu \cup \nu$ (obtained by arranging the nonzero terms of $\mu$ and $\nu$ in decreasing order).  The \emph{transpose} $\lambda^\bt$ of a partition $\lambda$ is given by $(\lambda^\bt)_i = \#\{j \mid \lambda_j \ge i\}$.  Note that $(\mu + \nu)^\bt = \mu^\bt \cup \nu^\bt$.

A convenient way to visualize partitions and bipartitions is via diagrams of boxes. For a partition $\lambda$ we use the usual left-justified Young diagram where the parts of $\lambda$ give the number of boxes in each row, and the parts of $\lambda^\bt$ give the number of boxes in each column. For a bipartition $(\mu;\nu)$, following \cite{ah}, we put the Young diagrams of $\mu$ and $\nu$ `back to back', separated by a vertical `wall'; thus the diagram of $\mu+\nu$ is obtained by forgetting the wall and left-justifying the boxes. For example,
\[
(3,2,1,1) =
\begin{tableau}
\row{\c\c\c}
\row{\c\c}
\row{\c}
\row{\c}
\end{tableau}
\qquad\text{and}\qquad
((2,1);(3,2,1,1))=
\begin{tableau}
\row{\c\c\lc\c\c}
\row{\q\c\lc\c}
\row{\q\q\lc}
\row{\q\q\lc}
\end{tableau}
\]

Finally, to any partition $\lambda \in \cP_n$, we associate the quantity
\[
n(\lambda) = \sum_{i=1}^\infty (i-1)\lambda_i = \sum_{i=1}^\infty 
\binom{(\lambda^\bt)_i}{2}.
\]

\subsection{Signed partitions}

A \emph{signed partition} is a pair $(\lambda,\epsilon)$, where $\lambda$ is a partition, and $\epsilon: \{1, \ldots, \ell(\lambda)\} \to \{+,-\}$ is a function such that if $\lambda_i = \lambda_j$, $\epsilon(i) = {+}$, and $\epsilon(j) = {-}$ hold, then $i < j$.  A signed partition determines two \emph{subordinate partitions} $\lambda^{(+)}$ and $\lambda^{(-)}$ as follows. Define compositions $\lambda^{(+)}$ and $\widetilde{\lambda}^{(-)}$ by
\[
\lambda^{(+)}_i = \begin{cases}
\lceil \lambda_i/2 \rceil & \text{if $\epsilon(i) = +$,}\\
\lfloor \lambda_i/2 \rfloor & \text{if $\epsilon(i) = -$,}
\end{cases}
\qquad
\widetilde{\lambda}^{(-)}_i = \lambda_i-\lambda^{(+)}_i.
\]
Then $\lambda^{(+)}$ is a partition, and $\widetilde{\lambda}^{(-)}$ fails to be a partition exactly when there exist some $i<j$ such that $\lambda_i = \lambda_j$ is odd, $\epsilon(i) = {+}$, and $\epsilon(j) = {-}$. We define $\lambda^{(-)}$ to be the partition obtained by rearranging the parts of $\widetilde{\lambda}^{(-)}$ in decreasing order. The \emph{signature} of a signed partition $(\lambda,\epsilon)$ is the pair of integers $(|\lambda^{(+)}|, |\lambda^{(-)}|)$.  The set of all signed partitions of signature $(d,d')$ is denoted $\SP_{d,d'}$.

The visual interpretation is as follows. The signed partition $(\lambda,\epsilon)$ may be drawn as the Young diagram of $\lambda$ with the values of $\epsilon$ filled in along the first column, and then signs inserted in the rest of the diagram so that `$+$' and `$-$' alternate across rows. The condition on $\epsilon$ stipulates that among the rows of a certain length, those beginning with `$+$' come above those beginning with `$-$'. For instance, $((6,3,3,3,2,1), (-,+,-,-,-,+))$ would be drawn as
\[
\begin{tableau}
\row{.-.+.-.+.-.+}
\row{.+.-.+}
\row{.-.+.-}
\row{.-.+.-}
\row{.-.+}
\row{.+}
\end{tableau}
\]
Then, $\lambda^{(+)}$ is obtained by erasing all `$-$' boxes and left-justifying the remaining boxes, and $\lambda^{(-)}$ is defined analogously but possibly with the additional step of re-ordering the rows.  In our example, we have
\[
\lambda^{(+)} =
\begin{tableau}
\row{\c\c\c}
\row{\c\c}
\row{\c}
\row{\c}
\row{\c}
\row{\c}
\end{tableau}
\qquad\text{and}\qquad
\lambda^{(-)} =
\begin{tableau}
\row{\c\c\c}
\row{\c\c}
\row{\c\c}
\row{\c}
\row{\c}
\row{\q}
\end{tableau}
\]
The signature of $(\lambda,\epsilon)$ simply counts the `$+$' boxes and the `$-$' boxes.


\subsection{Signed quasibipartitions}

A \emph{quasipartition} is a composition $\lambda = (\lambda_1, \lambda_2, \ldots)$ satisfying $\lambda_i \ge \lambda_j-1$ whenever $i \le j$.  

A \emph{signed quasibipartition} is a triple $(\mu;\nu, \epsilon)$ where $\mu$ and $\nu$ are quasipartitions such that $\mu+\nu$ is a partition, and $(\mu+\nu,\epsilon)$ is a signed partition such that
\[ \epsilon(i) = \begin{cases} 
+,&\text{ if $\mu_i\geq 1$ is odd,}\\
-,&\text{ if $\mu_i\geq 2$ is even,}\\
-,&\text{ if $\mu_i=0$ and there is some $j<i$ such that $\nu_j=\nu_{i}-1$,}\\
&\text{ or $\mu_i=0$ and there is some $j>i$ such that $\mu_j=1$.}
\end{cases} \]
Note that we do not specify $\epsilon(i)$ if $\mu_i=0$ and there is no $j$ as above.
The \emph{signature} of $(\mu;\nu,\epsilon)$ is the signature of $(\mu+\nu,\epsilon)$.  The set of all signed quasibipartitions of signature $(d,d')$ is denoted $\SQ_{d,d'}$.


\begin{rmk} \label{rmk:johnsonconvention}
The definition of signed quasibipartition given here is equivalent to that of `striped $2$-bipartition' given by Johnson~\cite[Definition 4.1]{johnson}. The main difference is that where we would have $\mu_i=0$, $\nu_i=s\geq 1$, and $\epsilon(i)=+$, he would have $\mu_i=-1$, $\nu_i=s+1$, and $\epsilon(i)=+$. The restrictions we imposed on this case are equivalent to saying that the quasipartition inequalities continue to hold if one applies this shift. Johnson's convention achieves a uniform rule that $\epsilon(i)=+$ if and only if $\mu_i$ is odd, but at the cost of allowing $\mu$ to have negative parts.
\end{rmk}

As above, we can draw a signed quasibipartition $(\mu;\nu,\epsilon)$ as a pair of back-to-back diagrams of boxes with the values of $\epsilon$ entered in the leftmost box of each row, and with `$+$' and `$-$' alternating across rows.  The condition on $\epsilon$ implies that every box immediately to the left of the wall contains a `$+$'.  For instance, the signed quasibipartition $((2,3,2,1,2,0,1);(4,2,3,3,1,2),(-,+,-,+,-,-,+))$ would be
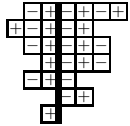
\[
\begin{tableau}
\row{\q.-\rce{+}.-.+.-.+}
\row{.+.-\rce{+}.-.+}
\row{\q.-\rce{+}.-.+.-}
\row{\q\q\rce{+}.-.+.-}
\row{\q.-\rce{+}.-}
\row{\q\q\q\lce{-}.+}
\row{\q\q\rce{+}}
\end{tableau}
\]

Given a signed quasibipartition $(\mu;\nu,\epsilon)$ with $\mu+\nu=\lambda$, we will define \emph{subordinate bipartitions} $(\mu^{(+)};\nu^{(+)})$ and $(\mu^{(-)};\nu^{(-)})$ such that
\begin{equation}
\mu^{(+)} + \nu^{(+)} = \lambda^{(+)}
\qquad\text{and}\qquad
\mu^{(-)} + \nu^{(-)} = \lambda^{(-)}.
\end{equation}
We first define quasipartitions $\widetilde{\mu}^{(+)},\widetilde{\nu}^{(+)},\widetilde{\mu}^{(-)},\widetilde{\nu}^{(-)}$ which count the number of boxes of a given sign on a given side of the wall and in a given row:  
\[
\begin{split}
\widetilde{\mu}_i^{(+)}&=\lceil\mu_i/2\rceil,\qquad\qquad\qquad\qquad\qquad\qquad\qquad
\widetilde{\mu}_i^{(-)}=\mu_i-\widetilde{\mu}_i^{(+)},\\
\widetilde{\nu}_i^{(+)}&=\begin{cases}
\lceil\nu_i/2\rceil&\text{ if $\mu_i=0$ and $\epsilon(i)=+$,}\\
\lfloor\nu_i/2\rfloor&\text{ otherwise,}
\end{cases}\;\qquad
\widetilde{\nu}_i^{(-)}=\nu_i-\widetilde{\nu}_i^{(+)}.
\end{split}
\]
Let $\lambda$ be the partition $\mu+\nu$.
Then $\widetilde{\mu}^{(+)}+\widetilde{\nu}^{(+)}=\lambda^{(+)}$ is a partition, but $\widetilde{\mu}^{(-)}+\widetilde{\nu}^{(-)}=\widetilde{\lambda}^{(-)}$ may not be, as seen above. If necessary, apply some permutation simultaneously to the parts of $\widetilde{\mu}^{(-)}$ and to the parts of $\widetilde{\nu}^{(-)}$ so that $\widetilde{\mu}^{(-)}+\widetilde{\nu}^{(-)}$ becomes the partition $\lambda^{(-)}$; it is easy to see that $\widetilde{\mu}^{(-)}$ and $\widetilde{\nu}^{(-)}$ will still be quasipartitions after this permutation. For example, starting from the above signed quasibipartition, we obtain
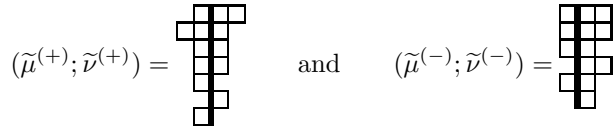
\[
(\widetilde{\mu}^{(+)};\widetilde{\nu}^{(+)})=
\begin{tableau}
\row{\q\c\lc\c}
\row{\c\c\lc}
\row{\q\c\lc}
\row{\q\c\lc}
\row{\q\rc}
\row{\q\q\lc}
\row{\q\rc}
\end{tableau}
\qquad\text{and}\qquad
(\widetilde{\mu}^{(-)};\widetilde{\nu}^{(-)})=
\begin{tableau}
\row{\c\lc\c}
\row{\c\lc\c}
\row{\c\lc}
\row{\q\lc\c}
\row{\c\lc}
\row{\q\lc}
\row{\q}
\end{tableau}
\]
To produce a bipartition $(\mu^{(+)};\nu^{(+)})$ from $(\widetilde{\mu}^{(+)};\widetilde{\nu}^{(+)})$, we apply the rectification procedure of \cite[Lemma 2.4]{ah}, which in the context of quasipartitions means that
\[
\mu_i^{(+)}=\begin{cases}
\widetilde{\mu}_i^{(+)}+1&\text{ if $\widetilde{\mu}_j^{(+)}=\widetilde{\mu}_i^{(+)}+1$ for some $j>i$}\\
&\text{ or $\widetilde{\nu}_j^{(+)}=\widetilde{\nu}_i^{(+)}-1$ for some $j<i$,}\\
\widetilde{\mu}_i^{(+)}&\text{ otherwise,}
\end{cases}
\] 
and $\nu_i^{(+)}=\lambda_i^{(+)}-\mu_i^{(+)}$.
We obtain $(\mu^{(-)};\nu^{(-)})$ from $(\widetilde{\mu}^{(-)};\widetilde{\nu}^{(-)})$ by the same rule. In our example, we have
\[
(\mu^{(+)};\nu^{(+)})=
\begin{tableau}
\row{\c\c\lc}
\row{\c\c\lc}
\row{\q\c\lc}
\row{\q\c\lc}
\row{\q\rc}
\row{\q\rc}
\row{\q\rc}
\end{tableau}
\qquad\text{and}\qquad
(\mu^{(-)};\nu^{(-)})=
\begin{tableau}
\row{\c\lc\c}
\row{\c\lc\c}
\row{\c\lc}
\row{\c\lc}
\row{\c\lc}
\row{\q\lc}
\row{\q}
\end{tableau}
\]

\subsection{Nilpotent orbits and enhanced nilpotent orbits}

Let $V$ be a complex vector space of dimension $d$, and let
\[
\cN_V = \{ x \in \End(V) \mid \text{$x$ is nilpotent} \}.
\]
(As in the introduction, we may omit the subscript from $\cN_V$ if only one vector space is involved.) $GL(V)$ acts on $\cN_V$ by conjugation.  The Jordan form theorem gives us the following well-known parametrization of orbits by partitions.
\begin{lem}
The $GL(V)$-orbits on $\cN_V$ are in bijection with $\cP_d$.  For $\lambda\in\cP_d$, an element $x\in\cN_V$ belongs to the orbit $\cO_\lambda$ if and only if there is a basis of $V$ which can be identified with the set of boxes in the diagram of $\lambda$, in such a way that $x$ sends a given box to the box on its left, or to zero if there is no box on its left.
\end{lem}
\noindent
If $x\in\cO_\lambda$, we refer to $\lambda$ as the \emph{Jordan type} of $x$.

The following characterization of the closures of these orbits is also well known. Recall the dominance partial order on partitions: if $\rho,\lambda \in \cP_d$, then $\rho \le \lambda$ if and only if for all $k$,
\[ \sum_{i=1}^k \rho_i \le \sum_{i=1}^k \lambda_i .\]

\begin{lem}\label{lem:nilp-crit}
For a nilpotent endomorphism $x \in \cN_V$, the following conditions are equivalent:
\begin{enumerate}
\item $x \in \overline{\cO_\lambda}$.
\item The Jordan type $\rho$ of $x$ satisfies $\rho\leq\lambda$.
\item $V$ admits a filtration $0 = V_0 \subset V_1 \subset \cdots \subset V_{\lambda_1} = V$ such that $x(V_i) \subset V_{i-1}$ and $\dim V_i/V_{i-1} = \lambda^\bt_{\lambda_1+1-i}$.
\item $V$ admits a filtration $0 = V_0 \subset V_1 \subset \cdots \subset V_{\lambda_1} = V$ such that $x(V_i) \subset V_{i-1}$ and $\dim V_i/V_{i-1} = \lambda^\bt_{i}$.
\end{enumerate}
\end{lem}

Next, the \emph{enhanced nilpotent cone} associated to $V$ is the variety $V \times \cN_V$.  The group $GL(V)$ acts on this cone with finitely many orbits as well. 
\begin{lem}[{\cite[Proposition~2.3]{ah}, \cite[Theorem 1]{travkin}}]
\label{lem:enh-param}
The $GL(V)$-orbits on $V\times\cN_V$ are in bijection with $\cQ_d$.  For $(\mu;\nu)\in\cQ_d$, an element $(v,x)\in V\times\cN_V$ belongs to the orbit $\cO_{\mu;\nu}$ if and only if there is a basis of $V$ which can be identified with the set of boxes in the diagram of $(\mu;\nu)$, in such a way that $x$ sends a given box to the box on its left, or to zero if there is no box on its left; and $v$ is the sum of the boxes immediately left of the wall.
\end{lem}
\noindent
Such a basis is called a \emph{normal basis} for $(v,x)$. If $(v,x)\in\cO_{\mu;\nu}$, we will refer to $(\mu;\nu)$ as the \emph{type} of $(v,x)$.

From this description, it is clear that
the projection map $\bar\pi^V: V \times \cN_V \to \cN_V$ satisfies
\begin{equation} \label{eqn:bipartitionsum}
\bar\pi^V(\cO_{\mu;\nu}) = \cO_{\mu+\nu}.
\end{equation}

Note that we can identify the ordinary nilpotent cone $\cN_V$ with the closed subvariety $\{0\}\times\cN_V$ of the enhanced nilpotent cone $V\times\cN_V$. Under this identification, the orbit $\cO_\lambda\subset\cN_V$ corresponds to the orbit $\cO_{\varnothing;\lambda}\subset V\times\cN_V$. Thus all our statements about enhanced nilpotent orbits and their closures will include a (usually well-known) statement about the unenhanced case. 

To state the analogue of Lemma~\ref{lem:nilp-crit}, we need the partial order on $\cQ_d$ defined as follows:
$(\rho; \sigma) \le (\mu; \nu)$ if and only if for all $k \ge 0$, 
\[
\begin{split}
\sum_{i=1}^k (\rho_i + \sigma_i) &\le
\sum_{i=1}^k (\mu_i + \nu_i),\text{ and}\\
\sum_{i=1}^k (\rho_i + \sigma_i) + \rho_{k+1} &\le 
\sum_{i=1}^k (\mu_i + \nu_i) + \mu_{k+1}.
\end{split}
\]

\begin{lem}[{\cite[Theorem~3.9, Corollary~3.4]{ah}}]\label{lem:enhcl-crit}
For $(v,x) \in V \times \cN_V$, the following conditions are equivalent:
\begin{enumerate}
\item $(v,x) \in \overline{\cO_{\mu;\nu}}$.
\item The type $(\rho;\sigma)$ of $(v,x)$ satisfies $(\rho;\sigma)\leq(\mu;\nu)$.
\item There is an $x$-stable $|\mu|$-dimensional subspace $W \subset V$ containing $v$ such that 
\begin{enumerate}
\item the Jordan type $\bar\mu$ of $x|_W$ satisfies $\bar\mu \le \mu$, and
\item the Jordan type $\bar\nu$ of $x|_{V/W}$ satisfies $\bar\nu \le \nu$.
\end{enumerate}
\end{enumerate}
\end{lem}
\noindent
Here, and subsequently, when $x$ is a nilpotent endomorphism of $V$ and $W$ is an $x$-stable subspace, $x|_W$ and $x|_{V/W}$ denote the induced nilpotent endomorphisms of $W$ and of $V/W$.

\subsection{Covering relations}
\label{sect:covering}

We have seen in Lemmas~\ref{lem:nilp-crit} and \ref{lem:enhcl-crit} that
the inclusion relations among ordinary (respectively, enhanced) nilpotent orbit closures correspond to a combinatorial partial order on the set of partitions (respectively, bipartitions). For later use, we recall the 
covering relations which generate these partial orders. Geometrically, these covering relations correspond to minimal degenerations of orbits.

It is well known that the covering relations $\lambda' < \lambda$ in the dominance partial order on $\cP_d$ are those in which a single box in the diagram for $\lambda$ moves down from an outside corner to the first available inside corner, resulting in the diagram of $\lambda'$:
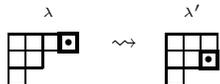
\[
\begin{tableau}
\caprow{\q\q\lambda}
\row{\q\q\q\q}
\row{\c\c\c\bce{\bullet}}
\row{\c\c\q\q}
\row{\c\q\q\q}
\end{tableau}
\quad\leadsto\,
\begin{tableau}
\caprow{\q\q\lambda'}
\row{\q\q\q\q}
\row{\q\c\c\c}
\row{\q\c\c\bce{\bullet}}
\row{\q\c\q\q}
\end{tableau}
\]

It is proved in \cite[Lemma 3.7]{ah} that there are 4 types of covering relations which generate the partial order on $\cQ_d$. We recall the pictorial description of these covering relations, putting the diagram of a bipartition $(\mu; \nu)$ on the left and the diagram of $(\mu'; \nu') < (\mu; \nu)$ on the right. In type (1), a single box moves down on the $\mu$ side of the dividing line, from an
outside corner to the first available inside corner, there being no inside or outside
corners on the $\nu$ side between these two positions:
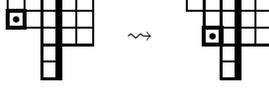
\[
\begin{tableau}
\row{\c\c\rc\c\c}
\row{\bce{\bullet}\c\rc\c\c}
\row{\q\q\rc\c\c}
\row{\q\q\rc}
\row{\q\q\rc}
\end{tableau}
\quad
\rightsquigarrow
\quad
\begin{tableau}
\row{\c\c\rc\c\c}
\row{\q\c\rc\c\c}
\row{\q\bce{\bullet}\rc\c\c}
\row{\q\q\rc}
\row{\q\q\rc}
\end{tableau}
\]
Type (2) is analogous, but with the box moving on the 
$\nu$ side of the dividing line:
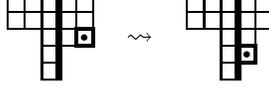
\[
\begin{tableau}
\row{\c\c\rc\c\c}
\row{\c\c\rc\c\c}
\row{\q\q\rc\c\bce{\bullet}}
\row{\q\q\rc}
\row{\q\q\rc}
\end{tableau}
\quad
\rightsquigarrow
\quad
\begin{tableau}
\row{\c\c\rc\c\c}
\row{\c\c\rc\c\c}
\row{\q\q\rc\c}
\row{\q\q\rc\bce{\bullet}}
\row{\q\q\rc}
\end{tableau}
\]
In type (3), a column of boxes (possibly a single box) moves directly to
the right, from an outside corner on the $\mu$ side to an inside corner
on the $\nu$ side:
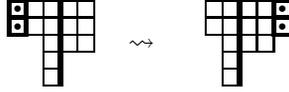
\[
\begin{tableau}
\row{\bce{\bullet}\c\rc\c\c}
\row{\bce{\bullet}\c\rc\c\c}
\row{\q\q\rc\c\c}
\row{\q\q\rc}
\row{\q\q\rc}
\end{tableau}
\quad
\rightsquigarrow
\quad
\begin{tableau}
\row{\q\c\rc\c\c\bce{\bullet}}
\row{\q\c\rc\c\c\bce{\bullet}}
\row{\q\q\rc\c\c}
\row{\q\q\rc}
\row{\q\q\rc}
\end{tableau}
\]
In type (4), a column of boxes (possibly a single box) moves to the left
and down one row, from an outside corner on the $\nu$ side to an inside corner
on the $\mu$ side:
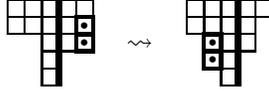
\[
\begin{tableau}
\row{\c\c\rc\c\c}
\row{\c\c\rc\c\bce{\bullet}}
\row{\q\q\rc\c\bce{\bullet}}
\row{\q\q\rc}
\row{\q\q\rc}
\end{tableau}
\quad
\rightsquigarrow
\quad
\begin{tableau}
\row{\c\c\rc\c\c}
\row{\c\c\rc\c}
\row{\q\bce{\bullet}\rc\c}
\row{\q\bce{\bullet}\rc}
\row{\q\q\rc}
\end{tableau}
\]

\subsection{Nilpotent pairs and enhanced nilpotent pairs}
\label{sect:nilppairs}

Now, let $V$ and $V'$ be complex vector spaces, say of dimensions $d, d'$, and let
\[
\cN_{V,V'} = \{ (x,y) \in \Hom(V,V') \times \Hom(V',V) \mid \text{$xy$ is nilpotent} \}.
\]
Note that $xy$ is nilpotent if and only if $yx$ is nilpotent, so we will make no distinction between $\cN_{V,V'}$ and $\cN_{V',V}$. Elements $(x,y)\in\cN_{V,V'}$ are known as \emph{nilpotent pairs}.  The group $GL(V) \times GL(V')$ acts on $\cN_{V,V'}$ with finitely many orbits.
\begin{lem}[{\cite[Section 4]{kp}}] \label{lem:signedpartition}
The $(GL(V) \times GL(V'))$-orbits in $\cN_{V,V'}$ are in bijection with $\SP_{d,d'}$.  For $(\lambda,\epsilon) \in \SP_{d,d'}$, a nilpotent pair $(x,y)\in\cN_{V,V'}$ belongs to the orbit $\cC_{\lambda,\epsilon}$ if and only if there is a basis of $V$ which can be identified with the set of \textup{`$+$'} boxes in the diagram of $(\lambda,\epsilon)$, and a basis of $V'$ which can be identified with the set of \textup{`$-$'} boxes, in such a way that $x$ sends a given \textup{`$+$'} box to the \textup{`$-$'} box on its left, or to zero if there is no box on its left, and $y$ sends a given \textup{`$-$'} box to the \textup{`$+$'} box on its left, or to zero if there is no box on its left.
\end{lem}
 
Consider the maps $\bar p_V^{V,V'}: \cN_{V,V'} \to \cN_V$ and $\bar p_{V'}^{V,V'}: \cN_{V,V'} \to \cN_{V'}$ given by
\[
\bar p_V^{V,V'}(x,y) = yx
\qquad\text{and}\qquad
\bar p_{V'}^{V,V'}(x,y) = xy.
\]
Recall that the signed partition $(\lambda,\epsilon)$ determines subordinate partitions $\lambda^{(+)}$ and $\lambda^{(-)}$. Using the basis interpretation of Lemma~\ref{lem:signedpartition}, it is easy to see that
\begin{equation}
\bar p_V^{V,V'}(\cC_{\lambda,\epsilon}) = \cO_{\lambda^{(+)}}
\qquad\text{and}\qquad
\bar p_{V'}^{V,V'}(\cC_{\lambda,\epsilon}) = \cO_{\lambda^{(-)}}.
\end{equation}

Next, we consider the variety $V \times \cN_{V,V'}$, known as the variety of \emph{enhanced nilpotent pairs}. (Note that this definition is asymmetric in $V$ and $V'$.) The group $GL(V) \times GL(V')$ acts on this variety with finitely many orbits. These have the following parametrization due to Johnson (recall Remark~\ref{rmk:johnsonconvention}), combining aspects of Lemmas~\ref{lem:enh-param} and \ref{lem:signedpartition}. 
\begin{lem}[{\cite[Corollary 4.13]{johnson}}] \label{lem:johnson}
The $(GL(V) \times GL(V'))$-orbits in $V \times \cN_{V,V'}$ are in bijection with $\SQ_{d,d'}$. For $(\mu;\nu,\epsilon) \in \SQ_{d,d'}$, the element $(v,x,y)\in V \times \cN_{V,V'}$ belongs to the orbit $\cC_{\mu;\nu,\epsilon}$ if and only if there is a basis of $V$ which can be identified with the set of \textup{`$+$'} boxes in the diagram of $(\mu;\nu,\epsilon)$, and a basis of $V'$ which can be identified with the set of \textup{`$-$'} boxes, in such a way that $x$ and $y$ move boxes to the left as in Lemma~\ref{lem:signedpartition}, and $v$ is the sum of the boxes immediately left of the wall.
\end{lem}
\noindent
Recall that by definition every box immediately left of the wall contains a `$+$', so corresponds to a basis element of $V$, as required.

We have maps $p_V^{V,V'}: V \times \cN_{V,V'} \to V \times \cN_V$ and $p_{V'}^{V,V'} : V \times \cN_{V,V'} \to V' \times \cN_{V'}$ given by
\[
p_V^{V,V'}(v, x, y) = (v, yx)
\qquad\text{and}\qquad
p_{V'}^{V,V'}(v, x, y) = (xv, xy).
\]
We also have the map $\pi^{V,V'}: V \times \cN_{V,V'} \to \cN_{V,V'}$ given by projection onto the second factor.  These maps have the expected compatibilities:
\begin{equation}
p_V^{V,V'}(\cC_{\mu;\nu,\epsilon}) = \cO_{\mu^{(+)};\nu^{(+)}},
\
p_{V'}^{V,V'}(\cC_{\mu;\nu,\epsilon}) = \cO_{\mu^{(-)};\nu^{(-)}},
\
\pi^{V,V'}(\cC_{\mu;\nu,\epsilon}) = \cC_{\mu+\nu;\epsilon}.
\end{equation}
Of course, one could consider enhanced nilpotent pair orbits in $V' \times \cN_{V,V'}$ instead, and thus define maps $p_{V'}^{V',V}$, $p_V^{V',V}$, and $\pi^{V',V}$. In this case, one must remember to reverse the meaning of the signs, so that the `$+$' label is associated with $V'$ and the `$-$' label with $V$.

\subsection{Orbit dimensions}

Let $\lambda \in \cP_d$.  The dimension of the $GL(V)$-orbit $\cO_\lambda \subset \cN_V$ is given by the following well-known formula:
\begin{equation}\label{eqn:dim-orb}
\dim \cO_\lambda = d^2 - d - 2n(\lambda).
\end{equation}
Next, let $(\mu;\nu) \in \cQ_d$.  From~\cite[Proposition~2.8]{ah}, we have
\begin{equation}\label{eqn:dim-norb}
\dim \cO_{\mu;\nu} = d^2 - d - 2n(\mu+\nu) + |\mu| = \dim \bar\pi^V(\cO_{\mu;\nu}) + |\mu|.
\end{equation}
Moreover, the `extra' term $|\mu|$ has the following interpretation: for a point $(v,x) \in \cO_{\mu;\nu}$, let
\begin{equation}\label{eqn:end-cent}
E^x = \{ a \in \End(V) \mid ax = xa \}.
\end{equation}
Then the subspace $E^xv \subset V$ has dimension $|\mu|$.

\begin{lem}\label{lem:dim-norb}
If $\cO_{\rho;\sigma} \subset \overline{\cO_{\mu;\nu}}$, then we have:
\begin{align*}
\dim \cO_{\rho;\sigma} + |\rho| &\le \dim \cO_{\mu;\nu} + |\mu|, \\
|\rho| - n(\rho+\sigma) &\le |\mu| - n(\mu + \nu).
\end{align*}
\end{lem}
\begin{proof}
The two inequalities are equivalent to one another by~\eqref{eqn:dim-norb} (the difference between the left and right hand sides in the first statement is double that in the second statement).
It suffices to prove this in the case where $\cO_{\rho;\sigma}$ is a minimal degeneration of $\cO_{\mu;\nu}$.  Recall that the minimal degenerations of enhanced nilpotent orbits were given in Section~\ref{sect:covering} in terms of four kinds of `moves' applied to the bipartition $(\mu;\nu)$.  Assume that $(\rho;\sigma)$ is obtained from $(\mu;\nu)$ by such a move.  If the move is of type~(1) or~(2), we have $|\rho| = |\mu|$, so the first inequality holds trivially. In a move of type~(3),   
we have $\rho+\sigma=\mu+\nu$ and $|\rho|<|\mu|$, so the second inequality holds.
In a move of type~(4), we have $\rho_i+\sigma_{i-1}=\mu_i+\nu_{i-1}$ for all $i$
(interpreting $\sigma_0$ and $\nu_0$ as zero). Since
\[
|\rho|-n(\rho+\sigma)=\sum_{i=1}^\infty (2-i)(\rho_i+\sigma_{i-1})
\]
and likewise for $(\mu;\nu)$, both inequalities hold with equality in this case.
\end{proof}

\begin{rmk}
\label{rmk:lem-inequality}
In Section~\ref{sect:stratdim} we will use the second inequality of Lemma~\ref{lem:dim-norb} in a crucial way. To that end, we remark here that in moves of type (1) and (2), the difference between left and right hand sides in the second inequality is at least 1. It is exactly 1 when a single box moves
down to the row directly below. In moves of type (3), the difference is the number of boxes which move to the right. In moves of type (4), the difference is zero.
\end{rmk}

Next, for a signed partition $(\lambda,\epsilon)$, a formula for $\dim \cC_{\lambda,\epsilon}$ is given in~\cite[Proposition~5.3]{kp}. We will not use that formula itself, but only the upper bound
\begin{equation}\label{eqn:dim-np}
\dim \cC_{\lambda,\epsilon} \le \half(\dim \cO_{\lambda^{(+)}} + \dim \cO_{\lambda^{(-)}}) + dd',
\end{equation}
obtained by omitting a term that always takes nonpositive values. At one point, we will need the further fact that equality holds in \eqref{eqn:dim-np} if and only if no rearrangement of parts is necessary in forming the subordinate partition $\lambda^{(-)}$ from $(\lambda,\epsilon)$.

Lastly, consider an orbit $\cC_{\mu;\nu,\epsilon} \subset V \times \cN_{V,V'}$, and choose a point $(v,x,y) \in \cC_{\mu;\nu,\epsilon}$.  Define the set $E^{(x,y)} \subset \End(V) \oplus \End(V')$ by
\[
E^{(x,y)} = \{(a,b) \in \End(V) \oplus \End(V') \mid \text{$xa = bx$ and $ay = yb$} \}.
\]
Let $E^{(x,y)}_V$ and $E^{(x,y)}_{V'}$ denote the projections of $E^{(x,y)}$ to $\End(V)$ and to $\End(V')$, respectively.  As noted in~\cite[Proposition 5.2]{johnson}, we have an analogue of \eqref{eqn:dim-norb}:
\begin{equation}\label{eqn:dim-johnson}
\dim \cC_{\mu;\nu,\epsilon} = \dim \pi^{V,V'}(\cC_{\mu;\nu,\epsilon}) + \dim E^{(x,y)}_V v.
\end{equation}

\begin{lem}\label{lem:enpdim}
Consider an enhanced nilpotent pair orbit $\cC = \cC_{\mu;\nu,\epsilon}\subset V \times \cN_{V,V'}$.  For brevity, let $(\rho;\sigma) = (\mu^{(+)};\nu^{(+)})$ and $(\rho';\sigma') = (\mu^{(-)};\nu^{(-)})$, so that
\[
\cO_{\rho;\sigma} = p_V^{V,V'}(\cC) \subset V \times \cN_V
\qquad\text{and}\qquad
\cO_{\rho';\sigma'} = p_{V'}^{V,V'}(\cC) \subset V' \times \cN_{V'}.
\]
Next, let
\[
\cP = \cO_{\rho+\sigma} = \bar\pi^{V}(\cO_{\rho;\sigma}) \subset \cN_V
\qquad\text{and}\qquad
\cP' = \cO_{\rho'+\sigma'} = \bar\pi^{V'}(\cO_{\rho';\sigma'}) \subset \cN_{V'}.
\]
The dimension of $\cC$ satisfies the following two inequalities:
\begin{align*}
\dim \cC &\le \dim \cO_{\rho;\sigma} + \half(\dim \cP' - \dim \cP) + dd', \\
\dim \cC &\le \dim \cO_{\rho';\sigma'} + \half(\dim \cP - \dim \cP') + |\rho| - |\rho'| + dd'.
\end{align*}
\end{lem}
\begin{proof}
Let $(v,x,y) \in \cC$.  It is easy to see that $E^{(x,y)}_V \subset E^{yx}$, where the latter is defined as in~\eqref{eqn:end-cent}, so that
\begin{equation}\label{eqn:dim-vsp}
\dim E^{(x,y)}_V  v \le \dim E^{yx}  v.
\end{equation}
By the remarks following~\eqref{eqn:dim-norb}, we have $\dim E^{yx} v = |\rho|$.  Next,~\eqref{eqn:dim-np} says that $\dim \pi^{V,V'}(\cC) \le \half(\dim \cP + \dim \cP') + dd'$.  Combining these observations with~\eqref{eqn:dim-johnson}, we obtain:
\[
\dim \cC \le \half(\dim \cP + \dim \cP') + dd' + |\rho|.
\]
Recall from~\eqref{eqn:dim-norb} that $\dim \cO_{\rho;\sigma} = \dim \cP + |\rho|$, and $\dim \cO_{\rho';\sigma'} = \dim \cP' + |\rho'|$.  Both inequalities in the lemma follow.
\end{proof}

\section{Invariant Theory for Enhanced Nilpotent Orbits and Pairs}
\label{sect:ivt}

In this section, we fix two vector spaces $V$ and $V'$ such that $\dim V' > \dim V$.  Let $r = \dim V' - \dim V$.  Given an enhanced nilpotent orbit closure $\overline{\cO_{\mu;\nu}} \subset V \times \cN_V$, there are two natural ways to construct from it a subvariety of the enhanced nilpotent cone $V' \times \cN_{V'}$: namely, we can form either $p_{V'}^{V,V'}((p_V^{V,V'})^{-1}(\overline{\cO_{\mu;\nu}}))$ or $p_{V'}^{V',V}((p_V^{V',V})^{-1}(\overline{\cO_{\mu;\nu}}))$, using the appropriate one of the diagrams:
\begin{gather*}
\setcounter{equation}{1}
\xymatrix{
V \times \cN_V &
V \times \cN_{V,V'} \ar[l]_{p_V^{V,V'}}\ar[r]^{p_{V'}^{V,V'}} &
V' \times \cN_{V'}} \tag{\arabic{section}.1$\Ib$} \label{eqn:itib}\\
\xymatrix{
V \times \cN_V &
V' \times \cN_{V,V'} \ar[l]_{p_V^{V',V}}\ar[r]^{p_{V'}^{V',V}} &
V' \times \cN_{V'}} \tag{\arabic{section}.1$I$} \label{eqn:iti}
\end{gather*}
(This use of ``$\Ib$'' and ``$I$'' in the labels will be compatible with notation to be introduced in Section~\ref{sect:eqv}.)  The goal of this section (see Proposition~\ref{prop:itq}) is to identify these subvarieties of $V' \times \cN_{V'}$.  We begin with the following result, an enhanced analogue of~\cite[Theorem~2.2]{kp}.

Here, and throughout the paper, we write $H \git X$ for $\Spec \C[X]^H$, where $X$ is an affine variety acted on by a reductive group $H$.
Since we are working over $\C$, an $H$-equivariant closed embedding $X\hookrightarrow Y$ induces a closed embedding $H \git X \hookrightarrow H \git Y$.


\begin{lem}\label{lem:fftit}
In both~\eqref{eqn:itib} and~\eqref{eqn:iti}, the right-hand map is an invariant-theoretic $GL(V)$-quotient map onto its image.  That is, $p_{V'}^{V,V'}$ induces an isomorphism
\begin{equation}
GL(V) \git (V \times \cN_{V,V'}) \overset{\sim}{\to} p_{V'}^{V,V'}(V \times \cN_{V,V'}),\label{eqn:itqib}
\end{equation}
and $p_{V'}^{V',V}$ induces an isomorphism
\begin{equation}
GL(V) \git (V' \times \cN_{V,V'}) \overset{\sim}{\to} p_{V'}^{V',V}(V' \times \cN_{V,V'}). \label{eqn:itqi}
\end{equation}
\end{lem}

\begin{proof}
A fundamental result of invariant theory states that for three vector spaces $U$, $U'$, $U''$, the composition map $m: \Hom(U,U') \times \Hom(U',U'') \to \Hom(U,U'')$ induces a closed embedding
\[
GL(U') \git (\Hom(U,U') \times \Hom(U',U'')) \hookrightarrow
\Hom(U,U'').
\]
See, for instance, \cite[Proposition 1.4c]{donkin}, which also shows that the image of this embedding consists of those linear maps in $\Hom(U,U'')$ whose rank is at most $\dim U'$.  Hence for any (reduced) $GL(U')$-stable closed subvariety $Y \subset \Hom(U,U') \times \Hom(U',U'')$, we get an isomorphism $GL(U') \git Y \overset{\sim}{\to} m(Y)$.

From this,~\eqref{eqn:itqi} follows by taking $U = U'' = V'$, $U' = V$, and $Y = \cN_{V,V'}$, since $GL(V)$ acts trivially on $V'$. We also get a description of the image:
\begin{equation}
p_{V'}^{V',V}(V' \times \cN_{V,V'})=V'\times\{z\in\cN_{V'}\,|\,\dim(\im(z))\leq\dim V\}.
\end{equation}
For~\eqref{eqn:itqib}, let $U = \C \times V'$, $U' = V$, and $U'' = V'$.  Then the result follows using the identifications
\[
\Hom(U,U') \cong V \times \Hom(V',V)
\qquad\text{and}\qquad \Hom(U,U'') \cong V' \times \End(V').
\]
We also get a description of the image:
\begin{equation}
p_{V'}^{V,V'}(V \times \cN_{V,V'})=\{(w,z)\in V'\times\cN_{V'}\,|\,\dim(\C w+\im(z))\leq\dim V\}.
\end{equation}
These formulas for the images of $p_{V'}^{V',V}$ and $p_{V'}^{V,V'}$ are the $(\mu;\nu)=(\dim V;\varnothing)$ special cases of Proposition~\ref{prop:itq} below.
\end{proof}

\begin{lem}\label{lem:filt-trans}
Let $x:V \to V'$ and $y:V' \to V$.  Let $\alpha = (a_1,a_2,\ldots, a_k)$ be a composition of $\dim V$, and suppose we have a filtration
\[
0 = V_0 \subset V_1 \subset \cdots \subset V_k = V
\]
such that $\dim V_i = a_1 + a_2 + \cdots + a_i$, and such that $yx(V_i) \subset V_{i-1}$.  Assume that $r \ge a_i$ for all $i$.  Fix $m \in \{1,\ldots,k+1\}$; if $m \le k$, assume that $a_m \ge a_{m+1} \ge \cdots \ge a_k$.  Then $V'$ admits a filtration
\[
0 = V'_0 \subset V'_1 \subset \cdots \subset V'_{k+1} = V
\]
such that $x(V_i) \subset V'_i$ and $y(V'_i) \subset V_{i-1}$, and such that
\[
\dim V'_i =
\begin{cases}
a_1 + \cdots + a_i & \text{if $i < m$,} \\
a_1 + \cdots + a_{i-1} + r & \text{if $i \ge m$.}
\end{cases}
\]
If, in addition, $r > a_m$, then for any vector $u' \in y^{-1}(V_{m-1})$, the filtration above may be chosen so that $u' \in V'_m$.
\end{lem}
\begin{proof}
Since $yx(V_i) \subset V_{i-1}$, we have $x(V_i) \subset y^{-1}(V_{i-1}) \subset y^{-1}(V_i)$.  Moreover:
\begin{align*}
\dim x(V_i) &\le a_1 + \cdots + a_i, \\
\dim y^{-1}(V_{i-1}) &\ge a_1 + \cdots + a_{i-1} + r.
\end{align*}
We will construct the spaces $V'_i$ by induction on $i$. Assume that we have already constructed subspaces $V'_0 \subset V'_1 \subset \cdots \subset V'_{i-1}$ with the desired properties.  Suppose first that $i < m$.  Since $x(V_{i-1}) \subset V'_{i-1}$, we see that
\[
\dim x(V_i)/(x(V_i) \cap V'_{i-1}) \le \dim V_i/V_{i-1} = a_i,
\]
and therefore
\[
\dim (V'_{i-1} + x(V_i)) \le a_1 + \cdots + a_{i-1} + a_i.
\]
Since $a_1 + \cdots + a_{i-1} + a_i \le a_1 + \cdots + a_{i-1} + r$, there exists a subspace $V'_i \subset V'$ such that
\[
V'_{i-1} + x(V_i) \subset V'_i \subset y^{-1}(V_{i-1})
\qquad\text{and}\qquad
\dim V'_i = a_1 + \cdots + a_i.
\]

In the case $i = m$, we proceed as above, except that we choose $V'_m$ to have dimension $a_1 + \cdots + a_{m-1} + r$.  For the last assertion of the lemma, note that
\[
\dim (V'_{m-1} + x(V_m) + \C u') \le a_1 + \cdots + a_m + 1.
\]
Provided that $r > a_m$,  we can choose $V'_m$ to satisfy the stronger condition that
\[
V'_{m-1} + x(V_m) + \C u' \subset V'_m \subset y^{-1}(V_{m-1})
\qquad\text{and}\qquad
\dim V'_m = a_1 + \cdots + a_{m-1} + r.
\]

Finally, for the case $i > m$, we now have $\dim V'_{i-1} = a_1 + \cdots + a_{i-2} + r$, so
\[
\dim (V'_{i-1} + x(V_i)) \le a_1 + \cdots + a_{i-2} + r + a_i
\]
Since $a_{i-1} \ge a_i$, we may choose a subspace $V'_i \subset V$ such that
\[
V'_{i-1} + x(V_i) \subset V'_i \subset y^{-1}(V_{i-1})
\qquad\text{and}\qquad
\dim V'_i = a_1 + \cdots + a_{i-1} + r,
\]
as desired.
\end{proof}

\begin{lem}\label{lem:jtype-trans}
Let $x:V \to V'$ and $y:V' \to V$ be linear maps such that $yx \in \End(V)$ is nilpotent.  Let $W \subset V$ be a subspace stable under $yx$.  Let $\mu \in \cP_{\dim W}$ and $\nu \in \cP_{\dim V/W}$, and assume that
\begin{enumerate}
\item the Jordan type $\bar\mu$ of $yx|_W$ satisfies $\bar\mu \le \mu$, and
\item the Jordan type $\bar\nu$ of $yx|_{V/W}$ satisfies $\bar\nu \le \nu$.
\end{enumerate}
If $r \ge \ell(\mu + \nu)$, then there exist $xy$-stable subspaces $W'$ and $\widetilde W'$ of $V'$  such that:
\begin{enumerate}
\item We have $x(W) \subset W' \subset \widetilde W' \subset y^{-1}(W)$.
\item $\dim W' = \dim W$, and $\dim \widetilde W' = \dim W + r$.
\item The Jordan types of the maps induced by $xy$ on various subquotients of $V'$ satisfy the following inequalities:
\begin{align*}
xy|_{W'}: &{}\le \mu &
xy|_{V/W'}: &{}\le \nu + 1^r \\
xy|_{\widetilde W'}: & {} \le \mu + 1^r &
xy|_{V/\widetilde W'}: &{} \le \nu
\end{align*}
\end{enumerate}
If, in addition, $r > \ell(\nu)$, then for any vector $u' \in y^{-1}(W)$, the space $\widetilde W'$ may be chosen so that $u' \in \widetilde W'$.
\end{lem}
\begin{proof}
By Lemma~\ref{lem:nilp-crit}, we can endow $W$ with a filtration
\[
0 = V_0 \subset V_1 \subset \cdots \subset V_{\mu_1} = W
\text{ where }
\dim V_i/V_{i-1} = \mu^\bt_{\mu_1+1-i}
\text{ and }yx(V_i)\subset V_{i-1}.
\]
We can likewise endow $V/W$ with a filtration with $\nu_1$ terms.  Let us lift this filtration to $V$ and denote its terms as follows:
\[
W = V_{\mu_1} \subset V_{\mu_1+1} \subset \cdots \subset V_{\mu_1+\nu_1}
\text{ where }
\dim V_i/V_{i-1} = \nu^\bt_{i-\mu_1}
\text{ and }yx(V_i)\subset V_{i-1}.
\]
Since $\nu^\bt_1 \ge \cdots \ge \nu^\bt_{\nu_1}$, $r\geq\mu^{\bt}_i$ for all $i$, and $r\geq\nu^\bt_i$ for all $i$, we can apply Lemma~\ref{lem:filt-trans} with $m=\mu_i+1$ to obtain a filtration
\[
0 = V'_0 \subset V'_1 \subset \cdots \subset V'_{\mu_1+\nu_1+1} = V'
\]
such that $x(V_i)\subset V_{i-1}'$, $y(V_i')\subset V_{i-1}$, and
\[
\dim V'_i/V'_{i-1} =
\begin{cases}
\mu^\bt_{\mu_1+1-i} & \text{if $i \le \mu_1$,} \\
r & \text{if $i = \mu_1 +1$,} \\
\nu^\bt_{i-\mu_1-1} & \text{if $i > \mu_1+1$.}
\end{cases}
\]
Let $W' = V'_{\mu_1}$, and let $\widetilde W' = V'_{\mu_1+1}$.  The Jordan-type assertions follow from Lemma~\ref{lem:nilp-crit}, and the last statement regarding a vector $u' \in y^{-1}(W)$ follows from the last statement of Lemma~\ref{lem:filt-trans}.
\end{proof}

%
%

We deduce the following enhanced analogue of~\cite[Lemma~2.3]{kp}. 

\begin{prop}\label{prop:itq}
Consider a $GL(V)$-orbit $\cO_{\mu;\nu} \subset V \times \cN_V$.  Assume that $r \ge \ell(\mu+\nu)$.
\begin{enumerate}
\item In the setting of~\eqref{eqn:itib}, we have $p_{V'}^{V,V'}((p_V^{V,V'})^{-1}(\overline{\cO_{\mu;\nu}})) = \overline{\cO_{\mu;\nu + 1^r}}$.  Indeed, $p_{V'}^{V,V'}$ induces an isomorphism
\[
GL(V) \git (p_V^{V,V'})^{-1}(\overline{\cO_{\mu;\nu}}) \overset{\sim}{\to} \overline{\cO_{\mu;\nu + 1^r}}.
\]
\item Assume furthermore that $r > \ell(\nu)$.  In the setting of~\eqref{eqn:iti}, we have $p_{V'}^{V',V}((p_V^{V',V})^{-1}(\overline{\cO_{\mu;\nu}})) = \overline{\cO_{\mu+1^r;\nu}}$. Indeed, $p_{V'}^{V',V}$ induces an isomorphism
\[
GL(V) \git (p_V^{V',V})^{-1}(\overline{\cO_{\mu;\nu}}) \overset{\sim}{\to} \overline{\cO_{\mu+1^r;\nu}}.
\]
\end{enumerate}
\end{prop}
\begin{proof}
For both parts of the proposition, the quotient statement follows from the determination of the image, by Lemma~\ref{lem:fftit}. 
Moreover, that lemma implies that $p_{V'}^{V,V'}((p_V^{V,V'})^{-1}(\overline{\cO_{\mu;\nu}}))$ and $p_{V'}^{V',V}((p_V^{V',V})^{-1}(\overline{\cO_{\mu;\nu}}))$ are closed in $V'\times\cN_{V'}$.

For part~(1), let $(v, x, y) \in (p_V^{V,V'})^{-1}(\overline{\cO_{\mu;\nu}})$.  Let $W \subset V$ be the subspace obtained by invoking Lemma~\ref{lem:enhcl-crit} for the pair $(v,yx)$.  In particular, $v \in W$.  Using Lemma~\ref{lem:jtype-trans}, we find a subspace $W' \subset V'$ containing $x(W)$ and, in particular, the vector $xv$.  The statements about Jordan type of $xy|_{W'}$ and $xy|_{V/W'}$ from that lemma, together with Lemma~\ref{lem:enhcl-crit}, show that $(xv, xy) \in \overline{\cO_{\mu;\nu+1^r}}$. Hence  
$p_{V'}^{V,V'}((p_V^{V,V'})^{-1}(\overline{\cO_{\mu;\nu}})) \subset \overline{\cO_{\mu;\nu + 1^r}}$. 
To prove the reverse inclusion, it suffices to show that for any $(v',z)\in\cO_{\mu;\nu + 1^r}\subset V'\times\cN_{V'}$, there exists some $(v, x, y) \in (p_V^{V,V'})^{-1}(\cO_{\mu;\nu})$ such that $xv=v'$ and $xy=z$. Since $\ell(\nu+1^r)=r\geq\ell(\mu)$, $\im(z)$ has dimension $\dim(V')-r=\dim V$ and contains $v'$. Fixing any vector space isomorphism $x:V\isomto\im(z)$, we can define $v$ and $y$ uniquely by the equations $xv=v'$ and $xy=z$, and it is easy to see that $(v,yx)\in\cO_{\mu;\nu}$, as desired. 

For part~(2), let $(v', x, y) \in (p_V^{V',V})^{-1}(\overline{\cO_{\mu;\nu}})$.  
Let $W \subset V$ be the subspace obtained by invoking Lemma~\ref{lem:enhcl-crit} for the pair $(yv',yx)$, so that $yv' \in W$.  Using the fact that $r > \ell(\nu)$, we may invoke Lemma~\ref{lem:jtype-trans} to find a subspace $\widetilde W' \subset V'$ containing $v'$.  The statements about Jordan type of $xy|_{\widetilde W'}$ and $xy|_{V/{\widetilde W'}}$ from that lemma, together with Lemma~\ref{lem:enhcl-crit}, show that $(v',xy) \in \overline{\cO_{\mu+1^r;\nu}}$. Hence $p_{V'}^{V',V}((p_V^{V',V})^{-1}(\overline{\cO_{\mu;\nu}})) \subset \overline{\cO_{\mu+1^r;\nu}}$.
To prove the reverse inclusion, it suffices to show that for any $(v',z)\in\cO_{\mu+1^r;\nu}\subset V'\times\cN_{V'}$, there exists some $(v', x, y) \in (p_V^{V',V})^{-1}(\cO_{\mu;\nu})$ such that $xy=z$. Since $\ell(\mu+1^r)=r>\ell(\nu)$, $\im(z)$ has dimension $\dim(V')-r=\dim V$. Fixing any vector space isomorphism $x:V\isomto\im(z)$, we can define $y$ uniquely by the equation $xy=z$, and it is easy to see that $(yv,yx)\in\cO_{\mu;\nu}$, as desired. 
\end{proof}

\section{Enhanced Quiver Varieties of Type $A$}
\label{sect:eqv}

Fix a positive integer $n$ and a bipartition $(\mu;\nu) \in \cQ_n$.  Form the partition $\lambda = \mu+\nu$, and let $t = \lambda_1=\ell(\lambda^\bt)$.  That is, $t$ is the largest part of $\lambda$, and it is the number of columns in the diagram of $\lambda$.  It will be convenient to refer to the lengths of these columns in increasing order, so we define
\[
0 < r_0 \le r_1 \le \cdots \le r_{t-1}
\qquad\text{by}\qquad
r_i = \lambda^\bt_{t-i}.
\]
Next, let $I\subset\{0,\cdots,t-1\}$ be the unique subset such that the $r_i$'s for $i\in I$ are the column-lengths of $\mu$ in non-decreasing order, and $r_i=r_{i+1}$ and $i+1\in I$ together imply $i\in I$. We define 
$\Ib=\{0,\cdots,t-1\}\setminus I$, so that the $r_i$'s for $i\in\Ib$ are the column-lengths of $\nu$ in non-decreasing order.  (Note that together, the sequence $(r_i)$ and the set $I$ determine the bipartition $(\mu;\nu)$.)
Let $U_i=\C^{r_0+\cdots+r_{i-1}}$ for $0\leq i\leq t$. Note that $U_0=0$, $U_t=\C^n$. It is primarily $U_t$ which we think of as the vector space $V$ in the definition of enhanced nilpotent orbits.

The notation and conventions of the preceding paragraph will remain in effect for the next three sections.  The aim of this section is to define the `enhanced quiver variety' associated to these data, and to prove that normality of that variety implies the normality of $\overline{\cO_{\mu;\nu}}$.  In the subsequent two sections, we will make progress on studying the normality of enhanced quiver varieties. Throughout, we are guided by the results of Kraft--Procesi~\cite{kp} in the unenhanced case, which in our framework is the special case where $I=\varnothing$ (that is, $\mu=\varnothing$). 

\subsection{Review of results of Kraft--Procesi}

We first recall the `classical' version of our variety, denoted $Z$ in \cite{kp}.
\begin{defn}
Let $\Lambda_\lambda=\Lambda_{(r_i)}$ be the affine variety
consisting of tuples $(A_i,B_i)$ of linear maps arranged as follows:
\[
\xymatrix{
U_0 \ar@/^/[r]^{A_0} &
U_1 \ar@/^/[l]^{B_0} \ar@/^/[r]^{A_1} &
U_2 \ar@/^/[l]^{B_{1}} \ar@/^/[r]^{A_2} &
\cdots \ar@/^/[l]^{B_{2}}
\ar@/^/[r]^{A_{t-2}} &
U_{t-1} \ar@/^/[l]^{B_{t-2}} \ar@/^/[r]^{A_{t-1}} &
U_t\ar@/^/[l]^{B_{t-1}}
}
\]
which satisfy the equations
\[ B_iA_i=A_{i-1}B_{i-1}\text{ (equation in $\End(U_i)$) for all $1\leq i\leq t-1$.} \]
\end{defn}
Since $A_0$ and $B_0$ are zero maps, the first equation says that $B_1A_1=0$; it
follows that $(A_{i-1},B_{i-1})$ is a nilpotent pair for all $1\leq i\leq t$.

Note that the group $\prod_{i=0}^{t}GL(U_i)$
acts on $\Lambda_{(r_i)}$ via
\[ (g_i).(A_i,B_i)=(g_{i+1}A_ig_i^{-1},g_iB_ig_{i+1}^{-1}). \]
Later we will be taking a quotient by the action of
$H=\prod_{i=0}^{t-1}GL(U_i)$, but retaining the action of $GL(U_t)=GL_n(\C)$;
the vector space $U_t$ is in that sense on a different footing from the other $U_i$'s.

\begin{rmk}
In the context of quiver varieties,
$\Lambda_{(r_i)}$ is a special case of Nakajima's variety of quadruples
satisfying the ADHM equation, where the Dynkin diagram is that of type $A_{t-1}$.
In the notation of \cite{maffei}, $\Lambda_{(r_i)}=\Lambda(d,v)$ where
$d=(0,\cdots,0,n)$ and $v=(r_0,r_0+r_1,\cdots,r_0+\cdots+r_{t-2})$. 
\end{rmk}

We also introduce notation for the `naively expected dimension' of $\Lambda_{(r_i)}$,
i.e.\ the number of coordinates of the variables $A_i,B_i$ minus the number
of equations in those coordinates in the definition of the variety:
\[ 
\begin{split}
d_{(r_i)}&=2\sum_{i=0}^{t-1}(\dim U_i)(\dim U_{i+1})
-\sum_{i=1}^{t-1}(\dim U_i)^2\\
&=\sum_{i=1}^{t-1}(r_0+\cdots+r_{i-1})^2+2\sum_{0\leq i<j\leq t-1}r_ir_j.
\end{split}
\]

\begin{exam} \label{exam:tequals2}
When $t=2$ (and ignoring the zero maps), 
the variety $\Lambda_{(r_0,r_1)}$ consists of pairs of maps
\[
\xymatrix{
U_1 \ar@/^/[r]^{A_1} &
U_2 \ar@/^/[l]^{B_{1}}}\text{ such that }B_1A_1=0.
\]
Note that the kernel of any $B_1:U_2\to U_1$ 
has dimension at least $\dim U_2-\dim U_1=r_1$, and by assumption $r_1\geq r_0=\dim U_1$.
It follows that the pairs $(A_1,B_1)$ where
$A_1$ is injective form a dense open subvariety of $\Lambda_{(r_0,r_1)}$. This
open subvariety is a fibre bundle over the 
Grassmannian $\Gr_{r_0}(U_2)$, with fibres isomorphic to
$GL_{r_0}(\C)\times\Hom(\C^{r_1},\C^{r_0})$. So $\Lambda_{(r_0,r_1)}$ is
irreducible of dimension $r_0^2+2r_0r_1=d_{(r_0,r_1)}$. 
\end{exam}

Kraft and Procesi proved in general that $\Lambda_{(r_i)}=\Lambda_\lambda$ 
is not just irreducible:
\begin{thm}[{\cite[Theorem 3.3]{kp}}]\label{thm:kp}
$\Lambda_{(r_i)}$ is a normal complete intersection of dimension $d_{(r_i)}$.
\end{thm}

\begin{rmk}
The conventions imposed at the beginning of the section imply that $r_0\leq r_1\leq\cdots\leq r_{t-1}$, and the theorem depends on this assumption.  If the $r_i$'s are not weakly increasing, $\Lambda_{(r_i)}$ may still be defined as above, but it may not
even be irreducible, let alone normal.  
\end{rmk}

\subsection{Enhanced quiver varieties}

Now we `enhance' $\Lambda_{(r_i)}$ by incorporating vectors which are related by the
linear maps, in a way determined by the subset $I\subset\{0,\cdots,t-1\}$. Roughly, the idea is that each nilpotent pair $(A_i,B_i)$ in the definition of $\Lambda_{(r_i)}$ should be replaced by an enhanced nilpotent pair; the question is which of the vector spaces $U_i$ and $U_{i+1}$ should be distinguished as the one containing the vector. Proposition~\ref{prop:itq} tells us that if $i\in\Ib$, meaning that the corresponding column belongs to the $\nu$ side of the bipartition, the vector should belong to $U_i$; and if $i\in I$, meaning that the corresponding column belongs to the $\mu$ side of the bipartition, the vector should belong to $U_{i+1}$. In each case we obtain a vector in the other vector space by applying the appropriate map. There is then a natural consistency condition when the enhanced nilpotent pairs are assembled together, resulting in the following definition.

\begin{defn}
Let $\Lambda_{\mu;\nu}=\Lambda_{(r_i),I}$ be the closed subvariety of
$(\prod_{i=0}^t U_i)\times\Lambda_{(r_i)}$ consisting of those $(u_i,A_i,B_i)$ which
satisfy the equations
\[
\begin{split}
B_{i}u_{i+1}&=u_{i}\text{ (equation in $U_{i}$), for all $i\in I$, and}\\
A_{i}u_{i}&=u_{i+1}\text{ (equation in $U_{i+1}$), for all 
$i\in \Ib$.} 
\end{split}
\]
\end{defn}
Since $u_0=0$, the second equation implies that $u_1=\cdots=u_{k}=0$, where
$k$ is the minimal element of $I$; or that $u_1=\cdots=u_t=0$, in the case that
$I=\varnothing$. More generally, the two equations imply that all $u_i$'s can be
determined from those indexed by 
$i\in (I+1)\setminus I$.
\begin{exam} \label{exam:112}
In the variety $\Lambda_{(1,1,2),\{0,2\}}$ attached to the bipartition $((2,1);(1))$, the equations satisfied by $u_0,u_1,u_2,u_3$ are $B_0 u_1=u_0$, $A_1 u_1=u_2$, and $B_2 u_3=u_2$. The first of these equations is automatic because $U_0=0$. Setting $u=u_1$ and $v=u_3$, we eliminate $u_2$ and get the single equation $A_1 u=B_2 v$, recovering 
Example \ref{exam:intro}.
\end{exam}

The action of $\prod_{i=0}^{t}GL(U_i)$ on $\Lambda_{(r_i)}$ extends to 
$\Lambda_{(r_i),I}$ in the obvious way:
\[ (g_i).(u_i,A_i,B_i)=(g_iu_i,g_{i+1}A_ig_i^{-1},g_iB_ig_{i+1}^{-1}). \]

The `naively expected dimension' of $\Lambda_{(r_i),I}$ is given by
\[
\begin{split}
d_{(r_i),I}&=d_{(r_i)}+\sum_{i=0}^{t} r_0+\cdots+r_{i-1} - \sum_{i\in I}
r_0+\cdots+r_{i-1} - \sum_{i\in \Ib} r_0+\cdots+r_{i}\\
&=d_{(r_i)}+\sum_{i\in I}r_{i}=d_{(r_i)}+|\mu|.
\end{split}
\]

Recall that the conventions in force imply that $r_0 \le r_1 \le \cdots \le r_{t-1}$, and also that whenever $r_i = r_{i+1}$ and $i+1 \in I$, we have $i \in I$ as well.  Theorem~\ref{thm:kp} is a special case (where $I = \varnothing$) of the following conjecture.

\begin{conj} \label{conj:lambdanormal}
$\Lambda_{(r_i),I}$ is a normal
complete intersection of dimension $d_{(r_i),I}$.
\end{conj}

\begin{exam}\label{exam:lambdanormal}
Suppose that $t=2$ as in Example \ref{exam:tequals2}.
For the four different possible $I$'s,
the equations required of $u_1\in U_1$ and $u_2\in U_2$ are as follows:
\[
\begin{array}{ll}
I=\emptyset:\; u_1=u_2=0,
&I=\{0\}:\; A_1u_1=u_2,\\
I=\{1\}:\; u_1=0,\,B_1u_2=0,
&I=\{0,1\}:\; B_1u_2=u_1.
\end{array}
\]
So $\Lambda_{(r_0,r_1),\emptyset}\cong\Lambda_{(r_0,r_1)}$,
$\Lambda_{(r_0,r_1),\{0\}}\cong U_1\times\Lambda_{(r_0,r_1)}$,
and $\Lambda_{(r_0,r_1),\{0,1\}}\cong U_2\times\Lambda_{(r_0,r_1)}$; in all these
cases Conjecture~\ref{conj:lambdanormal} is an immediate consequence of Theorem 
\ref{thm:kp}. If $I=\{1\}$, then by assumption we have $r_0<r_1$, and 
$\Lambda_{(r_0,r_1),\{1\}}$ may be proved to be
irreducible of dimension $d_{(r_0,r_1),\{1\}}$ by an argument similar to that
in Example \ref{exam:tequals2}, using the dense open subvariety consisting of
triples $(u_2,A_1,B_1)$ where $\dim(\im(A_1)+\C u_2)=r_0+1$. 
The normality of 
$\Lambda_{(r_0,r_1),\{1\}}$ will be proved later, as a special case of Theorem~\ref{thm:main}.
(If we allowed $r_0$ to equal $r_1$, we would find that $\Lambda_{(r_0,r_1),\{1\}}$
had two irreducible components.) 
\end{exam}

The following special case of Conjecture~\ref{conj:lambdanormal} is immediate from Kraft and Procesi's result.

\begin{thm} \label{thm:easy}
If $I=\{0,1,\cdots,s-1\}$ for some $0\leq s\leq t$, then $\Lambda_{(r_i),I}$ is a normal
complete intersection of dimension $d_{(r_i),I}$.
\end{thm}
\begin{proof}
When $I$ has this special form, the conditions on the $u_i$'s for $(u_i,A_i,B_i)\in\Lambda_{(r_i),I}$ are equivalent to
\[
u_i=\begin{cases}
B_{i+1}B_{i+2}\cdots B_{s-1}u_s,&\text{ if $i<s$,}\\
A_{i-1}A_{i-2}\cdots A_su_s,&\text{ if $i>s$.}
\end{cases}
\]
Hence $\Lambda_{(r_i),I}\cong U_s\times\Lambda_{(r_i)}$. Moreover,
\[
d_{(r_i),I}=d_{(r_i)}+\sum_{i=0}^{s-1}r_i=d_{(r_i)}+\dim U_s.
\]
So the result follows from Theorem~\ref{thm:kp}.
\end{proof}

In the next two sections we will make further progress on Conjecture~\ref{conj:lambdanormal}, culminating in the proof of a different special case in Theorem~\ref{thm:main}.

\subsection{Normality for enhanced nilpotent orbits}

As mentioned in the introduction, Conjecture~\ref{conj:lambdanormal} implies Conjecture~\ref{conj:enhnorm} by virtue of the following result.

\begin{thm}\label{thm:normrel}
If $\Lambda_{\mu;\nu} = \Lambda_{(r_i),I}$ is a normal variety, then $\overline{\cO_{\mu;\nu}}$ is normal.
\end{thm}

To prove Theorem~\ref{thm:normrel}, it suffices to exhibit $\overline{\cO_{\mu;\nu}}$ as an invariant-theoretic quotient of $\Lambda_{(r_i);I}$ by a reductive group, since passage to such a quotient preserves normality. The precise statement, generalizing the $I=\varnothing$ case proved by Kraft and Procesi \cite[Theorem 3.3]{kp}, is as follows.

\begin{thm}\label{thm:gitquot}
Let $H = GL(U_0) \times \cdots \times GL(U_{t-1})$.  Then the map
\[
\Phi:\Lambda_{(r_i),I}\to U_t\times\cN_{U_t}:(u_i,A_i,B_i)\mapsto(u_{t},A_{t-1}B_{t-1})
\]
has image $\overline{\cO_{\mu;\nu}}$ and induces an isomorphism
\[
H \git \Lambda_{(r_i),I} \isomto \overline{\cO_{\mu;\nu}}.
\]
\end{thm}
\begin{proof}
We proceed by induction on $t$.  If $t = 1$, then $H$ is the trivial group.  If $I = \varnothing$, then $\Lambda_{(r_i),I}$ consists of the single point where $u_0=u_1=0$ and $A_0=B_0=0$, and $\cO_{\mu;\nu} = \cO_{\varnothing;1^{r_0}}=\{(0,0)\}$, so the result holds.  On the other hand, if $I = \{0\}$, then $\Lambda_{(r_i),I}$ consists of the tuples $(u_0,u_1,A_0,B_0)$ where $u_0=0$, $A_0=B_0=0$, and $u_1$ is arbitrary, and $\overline{\cO_{\mu;\nu}} = \overline{\cO_{1^{r_0};\varnothing}}=U_1\times\{0\}$, so the result holds in this case as well.

Now, suppose $t > 1$, and let us put
\[
\Lambda' = \Lambda_{(r_0,r_1,\ldots,r_{t-2}), I \cap \{0,\ldots,t-2\}}.
\]
This is an enhanced quiver variety associated to a bipartition $(\mu';\nu')$ of total size $r_0 + \cdots + r_{t-2}$.  Let $H' = GL(U_0) \times \cdots \times GL(U_{t-2})$.  Then, by assumption, the map
\[
\Phi':\Lambda'\to U_{t-1}\times\cN_{U_{t-1}}:(u_i,A_i,B_i)\mapsto(u_{t-1},A_{t-2}B_{t-2})
\]
has image $\overline{\cO_{\mu';\nu'}}$ and induces an isomorphism
\[
H' \git \Lambda' \isomto \overline{\cO_{\mu';\nu'}}.
\]
Consider the variety
\[
Y =
\begin{cases}
(p_{U_{t-1}}^{U_{t-1},U_t})^{-1}(\overline{\cO_{\mu';\nu'}})\subset U_{t-1} \times \cN_{U_{t-1},U_t} & \text{if $t-1 \in \Ib$,} \\
(p_{U_{t-1}}^{U_{t},U_{t-1}})^{-1}(\overline{\cO_{\mu';\nu'}})\subset U_t \times \cN_{U_{t-1},U_t} & \text{if $t-1 \in I$.}
\end{cases}
\]
For simplicity, we omit the superscripts on the maps $p_{U_{t-1}}$ and $p_{U_t}$ which distinguish between the two cases. By Proposition~\ref{prop:itq}, we know in both cases that $p_{U_t}$ induces an isomorphism 
\[ 
GL(U_{t-1}) \git Y \isomto \overline{\cO_{\mu;\nu}}.
\]  
In the commutative diagram
\[
\xymatrix{
\Lambda_{(r_i),I} \ar[r]^\varphi \ar[d] & Y \ar[d]^{p_{U_{t-1}}}\ar[r]^{p_{U_t}} & \overline{\cO_{\mu;\nu}} \\
\Lambda' \ar[r]^{\Phi'} & \overline{\cO_{\mu';\nu'}} }
\]
the square on the left is cartesian, so $\varphi$ induces an isomorphism $H' \git \Lambda_{(r_i),I}\isomto Y$.  Since $\Phi=p_{U_t}\circ \varphi$, the result follows.
\end{proof}

\section{The Singular Locus of $\Lambda_{(r_i),I}$}
\label{sect:singloc}

We retain all the notation introduced in Section~\ref{sect:eqv}.  Let $\Lambda^\circ \subset \Lambda_{(r_i),I}$ be the open subset consisting of points $(u_i,A_i,B_i)$ such that:
\begin{equation}
\label{eq:lambda-zero}
\begin{cases}
\text{either 
$A_{j}$ is injective and $u_{j+1} \notin \im(A_j)$ or $B_{j}$ is surjective,} &\text{for all }j\in I, \\
\text{either 
$A_{j}$ is injective or $B_{j}$ is surjective,} &\text{for all }j\in \Ib.
\end{cases} 
\end{equation}
It is easy to see that $\Lambda^\circ$ is nonempty.


In this section, we prove that $\Lambda^\circ$ is nonsingular, and we use this to reframe Conjecture~\ref{conj:lambdanormal} as a dimension calculation.

Define a morphism of affine varieties
\[ \Psi:\prod_{i=0}^t U_i\times
\prod_{i=0}^{t-1}\left(\Hom(U_i,U_{i+1})\times\Hom(U_{i+1},U_i)
\right)
\to
\prod_{i\in I}U_{i}\times\prod_{i\in \Ib}U_{i+1}\times\prod_{i=1}^{t-1}\End(U_i) \]
by the rule
\[
\Psi(u_i,A_i,B_i)=
(B_{i}u_{i+1}-u_{i},A_{i}u_{i}-u_{i+1},B_iA_i-A_{i-1}B_{i-1}).
\]
Then $\Lambda_{(r_i),I}$ is the variety-theoretic zero fibre of $\Psi$.
Let $\Psi^{-1}(0)$ denote the scheme-theoretic zero fibre; in other words,
the spectrum of the quotient of the free polynomial ring in indeterminates
identified with the coordinates of the $u_i$'s, $A_i$'s and $B_i$'s by the
ideal generated by the coordinates of the appropriate vectors $B_{i}u_{i+1}-u_{i}$
and $A_{i}u_{i}-u_{i+1}$
and the matrices $B_iA_i-A_{i-1}B_{i-1}$. \textit{A priori}, 
$\Psi^{-1}(0)$ is a possibly
non-reduced scheme, whose reduced subscheme is the variety $\Lambda_{(r_i),I}$.

\begin{prop}\label{prop:nonsingloc}
$\Lambda^\circ$ is nonsingular of dimension $d_{(r_i),I}$.
\end{prop}
\begin{proof}
To prove the proposition, it suffices to show that for fixed $(u_i,A_i,B_i)\in \Lambda^\circ$,
the differential 
\[ 
\begin{split}
d\Psi_{(u_i,A_i,B_i)}:\bigoplus_{i=0}^t U_i\oplus\bigoplus_{i=0}^{t-1}
&\left(\Hom(U_i,U_{i+1})\oplus\Hom(U_{i+1},U_i)\right)\\
&\to
\bigoplus_{i\in I}U_{i}\oplus\bigoplus_{i\in \Ib}U_{i+1}
\oplus\bigoplus_{i=1}^{t-1}\End(U_i) 
\end{split}
\]
is surjective. This differential maps $(u_i',A_i',B_i')$ to
\[ 
(B_{i}'u_{i+1}+B_{i}u_{i+1}'-u_{i}',A_{i}'u_{i}+A_{i}u_{i}'-u_{i+1}',
B_i'A_i+B_iA_i'-A_{i-1}'B_{i-1}-A_{i-1}B_{i-1}').
\]
The proof of surjectivity is a slight elaboration of the proof of
\cite[Proposition 3.5]{kp}. We introduce filtrations of the domain and codomain:
\[
\begin{split}
&0=X_{t}\subset X_{t-1}\subset\cdots\subset X_0=\bigoplus_{i=0}^t U_i\oplus
\bigoplus_{i=0}^{t-1}\left(\Hom(U_i,U_{i+1})\oplus\Hom(U_{i+1},U_i)\right),\\
&\qquad\text{ where }X_j=\bigoplus_{i=j+1}^t U_i\oplus
\bigoplus_{i=j}^{t-1}\left(\Hom(U_i,U_{i+1})\oplus\Hom(U_{i+1},U_i)\right)\\
&\text{and }0=Y_{t}\subset Y_{t-1}\subset\cdots\subset Y_0=
\bigoplus_{i\in I}U_{i}\oplus\bigoplus_{i\in \Ib}U_{i+1}
\oplus\bigoplus_{i=1}^{t-1}\End(U_i),\\
&\qquad\text{ where }Y_j=
\bigoplus_{\substack{i\in I\\i\geq j}}U_{i}\oplus
\bigoplus_{\substack{i\in \Ib\\i\geq j}}U_{i+1}
\oplus
\bigoplus_{i=j}^{t-1}\End(U_i).
\end{split}
\]
It is immediate from the above formula that $d\Psi_{(u_i,A_i,B_i)}(X_j)\subset Y_j$
for all $j$, so it suffices to show that the induced map
$\psi_j:X_j/X_{j+1}\to Y_j/Y_{j+1}$ is surjective for all $0\leq j\leq t-1$.

If $j\in I$, then $\psi_j$ can be identified with the map
\[ 
\begin{split}
U_{j+1}\oplus\Hom(U_{j},U_{j+1})&\oplus\Hom(U_{j+1},U_{j})\to U_{j}\oplus\End(U_{j})\\
(u_{j+1}',A_{j}',B_{j}')&\mapsto
(B_{j}'u_{j+1}+B_{j}u_{j+1}',B_{j}'A_{j}+B_{j}A_{j}').
\end{split}
\]
This is surjective because, by the assumption that $(u_i,A_i,B_i)\in \Lambda^\circ$, 
either $B_{j}$ is surjective (allowing any image 
to be obtained by varying $u_{j+1}'$ and $A_{j}'$, with $B_{j}'$ set to zero)
or the matrix formed by adding
$u_{j+1}$ as an extra column to $A_{j}$ has full rank (allowing any image to be
obtained by varying $B_{j}'$, with $u_{j+1}'$ and $A_{j}'$ set to zero).

If $j\in \Ib$, then $\psi_j$ can be identified with the map
\[ 
\begin{split}
U_{j+1}\oplus\Hom(U_{j},U_{j+1})&\oplus\Hom(U_{j+1},U_{j})\to U_{j+1}
\oplus\End(U_{j})\\
(u_{j+1}',A_{j}',B_{j}')&\mapsto
(A_{j}'u_{j}-u_{j+1}',B_{j}'A_{j}+B_{j}A_{j}').
\end{split}
\]
This is surjective because
either $B_{j}$ is surjective (allowing any image 
to be obtained by varying $u_{j+1}'$ and $A_{j}'$, with $B_{j}'$ set to zero)
or $A_{j}$ is injective (allowing any image to be
obtained by varying $u_{j+1}'$ and $B_{j}'$, with $A_{j}'$ set to zero).  
\end{proof}

We end this section by showing how to reduce Conjecture~\ref{conj:lambdanormal} to the following dimension bound.

\begin{conj}\label{conj:dimbound}
We have $\dim (\Lambda_{(r_i),I} \smallsetminus \Lambda^\circ) \le d_{(r_i),I} - 2$.
\end{conj}

\begin{thm}\label{thm:nonsingrel}
If $\dim (\Lambda_{(r_i),I} \smallsetminus \Lambda^\circ) \le d_{(r_i),I} - 2$, then $\Psi^{-1}(0) \cong \Lambda_{(r_i),I}$, and this variety is a normal complete intersection of dimension $d_{(r_i),I}$.
\end{thm}
\begin{proof}
If we let $f(i)$ be the positive integer 
$t+1+i-2|I\cap\{0,\cdots,i-1\}|$ for $0\leq i\leq t$, then
\[
\begin{split}
\Psi(\lambda^{f(i)}u_i,&\lambda A_i,\lambda B_i)\\
&=(\lambda^{f(i)}(B_{i}u_{i+1}-u_{i}),\lambda^{f(i+1)}(A_{i}u_{i}-u_{i+1}),
\lambda^2(B_iA_i-A_{i-1}B_{i-1})).
\end{split}
\]
Hence $\Psi^{-1}(0)$ is connected, because it is a cone over a subscheme
of weighted projective space (with all weights positive).

From Proposition~\ref{prop:nonsingloc} and the assumption on $\dim (\Lambda_{(r_i),I} \smallsetminus \Lambda^\circ)$, we see that the scheme $\Psi^{-1}(0)$ is a connected complete intersection which is regular in
codimension $1$.  By Serre's criterion, $\Psi^{-1}(0)$ is reduced, irreducible and
normal; it therefore coincides with the variety $\Lambda_{(r_i),I}$, as desired.
\end{proof}

\section{Stratifications and Dimension Estimates}
\label{sect:stratdim}

In this section, we endow $\Lambda_{(r_i),I}$ with a stratification, and we estimate the dimension of each stratum.  This dimension estimate enables us to reduce Conjecture~\ref{conj:lambdanormal} to a purely combinatorial statement about sequences of signed quasibipartitions.

For $0 \le j \le t-1$, let us define
\[
\hat \cN_j =
\begin{cases}
U_j \times \cN_{U_j,U_{j+1}} & \text{if $j \in \Ib$,} \\
U_{j+1} \times \cN_{U_j,U_{j+1}} & \text{if $j \in I$,}
\end{cases}
\quad\text{and}\quad
\hat\cQ_j =
\begin{cases}
\SQ_{\dim U_j, \dim U_{j+1}} & \text{if $j \in \Ib$,} \\
\SQ_{\dim U_{j+1}, \dim U_j} & \text{if $j \in I$.}
\end{cases}
\]
Recall from Lemma~\ref{lem:johnson} that the $(GL(U_j) \times GL(U_{j+1}))$-orbits in $\hat \cN_j$ are parametrized by the set of signed quasibipartitions $\hat\cQ_j$.  There is an obvious map $h_j: \Lambda_{(r_i),I} \to \hat \cN_j$ which forgets all but the relevant vector and nilpotent pair, and we define a map
\[
\Theta: \Lambda_{(r_i),I} \to \prod_{j=0}^{t-1} \hat\cQ_j
\]
by associating to each point $(u_i,A_i,B_i) \in \Lambda_{(r_i),I}$ the sequence of signed quasibipartitions labelling the $(GL(U_j) \times GL(U_{j+1}))$-orbit of $h_j(u_i,A_i,B_i)$, for each $j$. Let $\Xi$ be the image of this map.  We endow $\Lambda_{(r_i),I}$ with a stratification indexed by $\Xi$ by taking the strata to be the fibres of the map above:
\[
\Lambda_{(r_i),I}^\xi = \Theta^{-1}(\xi),\text{ for any }\xi=(\xi_j)\in\Xi.
\]
It is clear that each $\Lambda_{(r_i),I}^\xi$ is a locally closed subvariety of $\Lambda_{(r_i),I}$ whose boundary is a union of
smaller such strata.  The strata are clearly preserved by the action of 
$\prod_{i=0}^t GL(U_i)$.

Each $\xi=(\xi_j)\in \Xi$ determines a sequence of bipartitions $(\rho_\xi^{(j)};\sigma_\xi^{(j)})$ for $1 \le j \le t$, by the following condition:
\[
(u_j, A_{j-1}B_{j-1}) \in \cO_{\rho_\xi^{(j)}; \sigma_\xi^{(j)}} \subset U_j \times \cN_{U_j}
\qquad
\text{for all $(u_i,A_i,B_i) \in \Lambda_{(r_i),I}^\xi$.}
\]
We also set $(\rho_\xi^{(0)};\sigma_\xi^{(0)}) = (\varnothing;\varnothing)$. Thus, for each $0\leq j\leq t-1$, $(\rho_\xi^{(j)};\sigma_\xi^{(j)})$ and $(\rho_\xi^{(j+1)};\sigma_\xi^{(j+1)})$ are the bipartitions subordinate to the signed quasibipartition $\xi_j$, with `$+$' corresponding to $U_j$ if $j\in\Ib$ and to $U_{j+1}$ if $j\in I$. The condition for $\xi$ to lie in $\Xi$ is exactly that the subordinate bipartitions of adjacent $\xi_j$'s match up in this way.

\begin{exam} \label{exam:strata}
Continue with Example~\ref{exam:112}. A stratum of $\Lambda_{(1,1,2),\{0,2\}}$ is indexed by a sequence of signed quasibipartitions $\xi=(\xi_0,\xi_1,\xi_2)$, where $\xi_0\in\SQ_{1,0}$, $\xi_1\in\SQ_{1,2}$, and $\xi_2\in\SQ_{4,2}$. The compatibility conditions these must satisfy are that the `$+$' subordinate bipartition of $\xi_0$ equals that of $\xi_1$ (this is $(\rho_\xi^{(1)};\sigma_\xi^{(1)})$), and that the `$-$' subordinate bipartition of $\xi_1$ equals that of $\xi_2$ (this is $(\rho_\xi^{(2)};\sigma_\xi^{(2)})$).
\end{exam}
 
\begin{prop}\label{prop:dimbound}
Each $\Lambda_{(r_i),I}^\xi$ is a smooth variety.  Moreover, we have
\begin{equation}\label{eqn:dim-stratum}
\dim \Lambda_{(r_i),I}^\xi \le d_{(r_i)} + n(\lambda) - n(\rho_\xi^{(t)}+\sigma_\xi^{(t)})
- \sum_{\substack{i \in I \\ i-1 \in \Ib}} |\rho_\xi^{(i)}| +
\sum_{\substack{i \in \Ib \cup \{t\} \\ i-1 \in I}} |\rho_\xi^{(i)}|.
\end{equation}
\end{prop}
\begin{proof}
For $0 \le j \le t-1$, let $\cC_j = h_j(\Lambda_{(r_i),I}^\xi)$.  This variety is, by definition, a single $(GL(U_j) \times GL(U_{j+1}))$-orbit in $\hat\cN_j$.  For $0 \le j \le t$, let
$\cO_j$ be the $GL(U_j)$-orbit in $U_j \times \cN_{U_j}$ labelled by the bipartition $(\rho_\xi^{(j)};\sigma_\xi^{(j)})$.  Finally, let $\cP_j = \bar\pi^{U_j}(\cO_j)$; this is the $GL(U_j)$-orbit in $\cN_{U_j}$ labelled by $\rho_\xi^{(j)} + \sigma_\xi^{(j)}$.

It is easy to see from the definition that $\Lambda_{(r_i),I}^\xi$ is isomorphic to the fibre product
\[
\cC_0 \times_{\cO_1} \cC_1 \times_{\cO_2} \cdots \times_{\cO_{t-1}} \cC_{t-1}.
\]
Since the varieties $\cC_i$ and the morphisms $\cC_i \to \cO_i$ and $\cC_i \to \cO_{i-1}$ are all smooth, it follows that $\Lambda^\xi_{(r_i),I}$ is smooth, and that its dimension is given by
\[
\dim \Lambda_{(r_i),I}^\xi = \sum_{i=0}^{t-1} \dim \cC_i - \sum_{i=1}^{t-1} \dim \cO_i.
\]
Since $\dim \cO_0 = 0$, there is no harm in changing this formula to
\begin{equation} \label{eqn:exactformula}
\dim \Lambda_{(r_i),I}^\xi = \sum_{i=0}^{t-1} (\dim \cC_i - \dim \cO_i).
\end{equation}
From Lemma~\ref{lem:enpdim}, we have that
\begin{multline*}
\dim \cC_i - \dim \cO_i \le \half(\dim \cP_{i+1} - \dim \cP_i) + (\dim U_i)(\dim U_{i+1}) \\
+ \begin{cases}
0 & \text{if $i \in \Ib$,} \\
|\rho_\xi^{(i+1)}| - |\rho_\xi^{(i)}| & \text{if $i \in I$.} 
\end{cases}
\end{multline*}
Summing up over $i \in \{0,1,\ldots, t-1\}$, we find that
\[
\dim \Lambda_{(r_i),I}^\xi \le
\half\dim \cP_t + \sum_{i=0}^{t-1} (\dim U_i)(\dim U_{i+1}) +
\sum_{\substack{i \in \Ib \cup \{t\} \\ i-1 \in I}} |\rho_\xi^{(i)}|
-\sum_{\substack{i \in I\\ i-1 \in \Ib}} |\rho_\xi^{(i)}|.
\]
The result then follows from the dimension formula~\eqref{eqn:dim-orb}, in the form
\[
\half\dim \cP_t = \half(\dim U_t)^2 - \half\dim U_t - n(\rho_\xi^{(t)} + \sigma_\xi^{(t)}),
\]
and the following calculation:
\begin{align*}
&\half(\dim U_t)^2 - \half\dim U_t + \sum_{i=0}^{t-1} (\dim U_i)(\dim U_{i+1}) \\
&= \half(r_0 + \cdots + r_{t-1})^2 - \half(r_0 + \cdots + r_{t-1})
+ \sum_{i=0}^{t-1} (r_0+\cdots +r_{i-1})(r_0 + \cdots+r_i) \\
&= \half\sum_{i=0}^{t-1} (r_i^2 - r_i) + \sum_{0\le i<j\le t-1} r_ir_j
+\sum_{i=0}^{t-1}(r_0+\cdots+r_{i-1})^2+ \sum_{0\le i<j\le t-1} r_ir_j \\
&= \sum_{i=0}^{t-1} \binom{r_i}{2} + d_{(r_i)} \\
&= d_{(r_i)} + n(\lambda). \qedhere
\end{align*}
\end{proof}

Recall from Theorem~\ref{thm:gitquot} that $\Phi(\Lambda_{(r_i),I}) = \overline{\cO_{\mu;\nu}}$, where $\Phi:(u_i,A_i,B_i)\mapsto(u_t,A_{t-1}B_{t-1})$. In particular, we have
\begin{equation} \label{eqn:dominance}
\cO_{\rho_\xi^{(t)};\sigma_\xi^{(t)}} =\Phi(\Lambda_{(r_i),I}^\xi)\subset \overline{\cO_{\mu;\nu}}
\qquad\text{and hence}\qquad
\cO_{\rho_\xi^{(t)}+\sigma_\xi^{(t)}} \subset \overline{\cO_{\lambda}}.
\end{equation}
By the dimension formula~\eqref{eqn:dim-orb}, this implies that $n(\rho_\xi^{(t)} + \sigma_\xi^{(t)}) \ge n(\lambda)$.  We have thus proved the following additional inequality.

\begin{cor} \label{cor:dimbound}
In the setting of Proposition~\ref{prop:dimbound}, we have
\[
\dim \Lambda_{(r_i),I}^\xi \le d_{(r_i)} 
- \sum_{\substack{i \in I \\ i-1 \in \Ib}} |\rho_\xi^{(i)}| +
\sum_{\substack{i \in \Ib \cup \{t\} \\ i-1 \in I}} |\rho_\xi^{(i)}|. \qed
\]
\end{cor}

The Kraft--Procesi stratification of $\Lambda_{(r_i)}$ is the $I=\varnothing$ special case of the stratification defined above. In this case, 
the last two terms in Proposition~\ref{prop:dimbound} and Corollary~\ref{cor:dimbound} vanish, and those results become dimension bounds obtained in~\cite[Section~5]{kp}. In fact, Kraft and Procesi proved the following more precise result.

\begin{thm}[{\cite[Section 5]{kp}}] \label{thm:kpstrata} Suppose that $I=\varnothing$.
\begin{enumerate}
\item For any $\xi$, $\dim\Lambda_{(r_i)}^{\xi}\leq d_{(r_i)}$.
\item Equality holds in \textup{(1)} 
for a unique $\xi$: the corresponding stratum consists
of those $(A_i,B_i)$ such that for all $i$, $A_i$ is injective and $B_i$ is surjective.
\item If $(A_i,B_i)$ belongs to a stratum $\Lambda_{(r_i)}^{\xi}$ of dimension
$d_{(r_i)}-1$, then for each $i$, either $A_i$ is injective or $B_i$ is surjective.
\end{enumerate}
\end{thm}

Motivated by this result, we formulate the following conjecture, which would clearly imply Conjecture~\ref{conj:dimbound} and hence Conjecture~\ref{conj:lambdanormal}.

\begin{conj} \label{conj:strata}
\begin{enumerate}
\item For any $\xi\in\Xi$, $\dim\Lambda_{(r_i),I}^{\xi}\leq d_{(r_i),I}$.
\item Equality holds in \textup{(1)} 
for a unique $\xi$: the corresponding stratum consists
of those $(u_i,A_i,B_i)$ such that $A_i$ is injective and $B_i$ is surjective for all $i$,
and $u_{i+1}\notin\im(A_i)$ for all $i\in I$.
\item If $\dim \Lambda^\xi_{(r_i),I} = d_{(r_i),I}-1$, then $\Lambda^\xi_{(r_i),I} \subset \Lambda^\circ$.
\end{enumerate}
\end{conj}

This formulation of the problem lends itself to purely combinatorial calculations.  The set of sequences of signed quasibipartitions $\Xi$ has a purely combinatorial description; the dimension upper bound in Proposition~\ref{prop:dimbound} is combinatorial in nature (and can easily be improved to an exact formula, at the cost of more combinatorial complexity); and the conditions in Conjecture~\ref{conj:strata}(2),(3) admit the following combinatorial descriptions.

\begin{lem} \label{lem:signed-diagram-interpretation}
Let $\xi=(\xi_j)\in\Xi$, and let $(u_i,A_i,B_i)\in\Lambda^\xi_{(r_i),I}$.
\begin{enumerate}
\item 
$(u_i,A_i,B_i)$ satisfies the condition of Conjecture~\ref{conj:strata}(2) if and only if, in the signed quasibipartitions $\xi_j$ for all $j\in\Ib$,
\[
\text{every row begins and ends with a }\textup{`$-$'}\text{ box,}
\] 
and in the signed quasibipartitions $\xi_j$ for all $j\in I$,
\[
\begin{array}{c}
\text{every row begins and ends with a }\textup{`$+$'}\text{ box}\\
\text{and there is a box immediately left of the wall which is at the end of a row.}
\end{array}
\]
\item
$(u_i,A_i,B_i)\in\Lambda^\circ$ if and only if, in the signed quasibipartitions $\xi_j$ for all $j\in\Ib$, 
\[
\begin{cases}
\text{either every row begins with a }\textup{`$-$'}\text{ box}\\
\text{or every row ends with a }\textup{`$-$'}\text{ box,}
\end{cases}
\]
and in the signed quasibipartitions $\xi_j$ for all $j\in I$,
\[
\begin{cases}
\text{either every row begins with a }\textup{`$+$'}\text{ box}\\
\quad\text{and there is a box immediately left of the wall which is at the end of a row}\\
\text{or every row ends with a }\textup{`$+$'}\text{ box.}
\end{cases}
\]
\end{enumerate}
\end{lem}
\begin{proof}
This is a straightforward translation of the definitions, using the basis interpretation of the signed quasibipartition diagram given in Lemma~\ref{lem:johnson}.
\end{proof}

The authors have implemented a computer program to test Conjecture~\ref{conj:strata} (and therefore all the other conjectures in the paper), and have found that it holds for all cases with $n \le 6$, with part (1) verified up to $n=9$.  

\begin{rmk}
Attempts to prove Conjecture~\ref{conj:strata} have revealed that not all properties which hold in \cite{kp} have obvious enhanced analogues. For example, in the Kraft--Procesi situation one has
\begin{equation}
\codim_{\Lambda_{(r_i)}}\,\Lambda^{\xi}_{(r_i)}\geq\frac{1}{2}\codim_{\overline{\cO_\lambda}}\,\Phi(\Lambda^{\xi}_{(r_i)}).
\end{equation}
This is one way of stating the $I=\varnothing$ case of Proposition~\ref{prop:dimbound}; the analogue for other classical groups is~\cite[Lemma 5.4]{kp2}. However, the obvious enhanced analogue of this inequality, namely
\begin{equation} \label{eqn:false}
\codim_{\Lambda_{(r_i),I}}\,\Lambda^{\xi}_{(r_i),I}\geq\frac{1}{2}\codim_{\overline{\cO_{\mu;\nu}}}\,\Phi(\Lambda^{\xi}_{(r_i),I}),
\end{equation}
is false in general. 
\end{rmk}
\begin{exam}
Continue with Example~\ref{exam:strata}. There is a stratum $\Lambda_{(1,1,2),\{0,2\}}^\xi$ consisting of all tuples $(u_i,A_i,B_i)$ such that $u_1\neq 0$, $A_1=0$, $B_1\neq 0$, $A_2$ is injective, and $B_2$ is surjective. The corresponding signed quasibipartitions are as follows:
\[
\xi_0 =
\begin{tableau}
\row{\rce{+}}
\end{tableau}
\qquad
\xi_1 =
\begin{tableau}
\row{\rce{+}\cm}
\row{\q\lce{-}}
\end{tableau}
\qquad
\xi_2 =
\begin{tableau}
\row{\rce{+}\cm\cp}
\row{\rce{+}\cm\cp}
\end{tableau}
\]
This stratum has codimension $1$ in $\Lambda_{(1,1,2),\{0,2\}}$, and belongs to $\Lambda^\circ$ in accordance with Conjecture~\ref{conj:strata}(3). However, $\Phi(\Lambda_{(1,1,2),\{0,2\}}^\xi)=\cO_{(1,1);(1,1)}$ has codimension $3$ in $\overline{\cO_{(2,1);(1)}}$, in violation of \eqref{eqn:false}.
\end{exam}

We now prove our conjectures in the special case which is `opposite' to the one handled in Theorem~\ref{thm:easy}.

\begin{thm}\label{thm:main}
If $I=\{s,s+1,\cdots,t-1\}$ for some $1\leq s\leq t-1$, then Conjecture~\ref{conj:strata} holds.
\end{thm}
\begin{proof}
First of all, note that the assumption that $I=\{s,s+1,\cdots,t-1\}$ for some $1\leq s\leq t-1$ is equivalent to saying that every column of $\mu$ is strictly longer than every column of $\nu$ (and both $\mu$ and $\nu$ are nonempty).

We prove the three parts of Conjecture~\ref{conj:strata} in turn.  Since $\Ib=\{0,1,\ldots, s-1\}$, we have $u_0 = u_1 = \cdots = u_s = 0$ for all $(u_i,A_i,B_i) \in \Lambda_{(r_i),I}$.  Therefore, for any $\xi \in \Xi$, we have $\rho^{(0)}_\xi = \rho^{(1)}_\xi = \cdots = \rho^{(s)}_\xi = \varnothing$.  The inequality in Proposition~\ref{prop:dimbound} reduces to
\begin{equation}
\label{eq:apply-prop-6.1}
\dim \Lambda_{(r_i),I}^\xi \le d_{(r_i)} + n(\lambda) - n(\rho^{(t)}_\xi + \sigma^{(t)}_\xi) + |\rho^{(t)}_\xi|.
\end{equation}
Recall from \eqref{eqn:dominance} that $\cO_{\rho^{(t)}_\xi;\sigma^{(t)}_\xi} \subset \overline{\cO_{\mu;\nu}}$.  It then follows from Lemma~\ref{lem:dim-norb} that
\begin{equation}
\label{eq:apply-lem-2.4}
d_{(r_i)} + n(\lambda) - n(\rho^{(t)}_\xi + \sigma^{(t)}_\xi) + |\rho^{(t)}_\xi| \le d_{(r_i)} + |\mu| = d_{(r_i),I}.
\end{equation}
Combining \eqref{eq:apply-prop-6.1} and \eqref{eq:apply-lem-2.4}, we deduce part~(1) of Conjecture~\ref{conj:strata}.

To prove part~(2) of Conjecture~\ref{conj:strata}, suppose that $\xi\in\Xi$ is such that $\dim \Lambda_{(r_i),I}^\xi = d_{(r_i),I}$. Then equality must hold in \eqref{eq:apply-lem-2.4}, our application of Lemma~\ref{lem:dim-norb}, which implies (see Remark~\ref{rmk:lem-inequality}) that $(\rho_\xi^{(t)};\sigma_\xi^{(t)})$ is obtained from $(\mu;\nu)$ by a sequence of type (4) moves. However, the assumption on the column lengths of $\mu$ and $\nu$ makes a type (4) move from $(\mu;\nu)$ impossible, and we conclude that $(\rho_\xi^{(t)};\sigma_\xi^{(t)})=(\mu;\nu)$. In particular, the number of rows of $\xi_{t-1}$ containing a `$+$' box is $\ell(\mu+\nu)=\ell(\mu)=r_{t-1}$. But $r_{t-1}=\dim U_t-\dim U_{t-1}$ is also the difference between the number of `$+$' boxes and the number of `$-$' boxes. Hence $\xi_{t-1}$ must be the signed quasibipartition obtained by labelling every box of the diagram of $(\mu;\nu)$ as `$+$', and then inserting a `$-$' box between any two adjacent `$+$' boxes in the same row (there is no ambiguity about the position of `$-$' boxes adjacent to the wall, since by definition every box immediately left of the wall must be `$+$'). From this, one deduces that $(\rho_\xi^{(t-1)};\sigma_\xi^{(t-1)})=(\mu';\nu)$, where $\mu'$ is obtained from $\mu$ by deleting the longest column. See the top half of Figure~\ref{fig:simul-degen} below for an example of a triple $(\mu;\nu),\xi_{t-1},(\mu';\nu)$ of this form (where $\xi_{t-1}$ is the top signed quasibipartition).

Repeating this argument, we obtain for all $j\geq s$ that $\sigma_\xi^{(j)}=\nu$ and $\rho_\xi^{(j)}$ is obtained from $\mu$ by deleting the $t-j$ longest columns. Moreover, for all $j\geq s$, $\xi_j$ is uniquely determined: it must be the signed quasibipartition obtained by labelling every box of the diagram of $(\rho_\xi^{(j+1)};\sigma_\xi^{(j+1)})$ as `$+$'  and then inserting `$-$' boxes as before. In particular, every row begins and ends with `$+$', and the column of `$+$' boxes immediately left of the wall is longer than the column of `$-$' boxes immediately right of the wall, as required by Lemma~\ref{lem:signed-diagram-interpretation}(1). The corresponding statements for $j<s$ follow in exactly the same way, where now, because $\rho_\xi^{(j)}=\varnothing$ for $j\leq s$, we have reverted to the unenhanced case as in Kraft and Procesi's proof of Theorem~\ref{thm:kpstrata}(2): one finds that $\sigma_\xi^{(j)}$ is obtained from $\nu$ by deleting the $s-j$ longest columns, and $\xi_j$ is obtained by labelling every box of $\sigma_\xi^{(j+1)}$ as `$-$' and then inserting `$+$' boxes (there are, of course, no boxes to the left of the wall). In particular, every row begins and ends with `$-$', as required by Lemma~\ref{lem:signed-diagram-interpretation}(1). As in \cite[Section 5.4]{kp}, it is easy to see that this is the only sequence of compatible signed quasibipartitions which satisfies the properties required by Lemma~\ref{lem:signed-diagram-interpretation}(1), so this stratum does have the description claimed in part~(2) of Conjecture~\ref{conj:strata}. In view of that description, it follows immediately from Theorem~\ref{thm:kpstrata}(2) that this stratum indeed has dimension $d_{(r_i),I}$. This completes the proof of part~(2) of Conjecture~\ref{conj:strata}; in fact, we now know the extra information that the only stratum for which equality holds in \eqref{eq:apply-lem-2.4} is the stratum we have just described.

To prove part~(3) of Conjecture~\ref{conj:strata}, let $\xi\in\Xi$ be such that $\dim \Lambda_{(r_i),I}^\xi = d_{(r_i),I} - 1$. Then equality cannot hold in \eqref{eq:apply-lem-2.4}, so it must be that equality holds in \eqref{eq:apply-prop-6.1} and fails by exactly 1 in \eqref{eq:apply-lem-2.4}. We aim to prove by induction on $t$ that this implies the conditions required by Lemma~\ref{lem:signed-diagram-interpretation}(2).
 
Let $\xi' = (\xi_0, \cdots, \xi_{t-2})$, $(r'_i) = (r_0, \cdots, r_{t-2})$, and $I' = \{s,s+1, \cdots, t-2\}$. (Note that $I'$ is empty if $s=t-1$.) The pair $((r_i'),I')$ corresponds to the bipartition $(\mu';\nu)$, where $\mu'$, as above, is obtained by deleting the longest column of $\mu$. From the proof of Proposition~\ref{prop:dimbound} it is clear that equality in \eqref{eq:apply-prop-6.1} forces the corresponding equality:
\begin{equation}
\dim \Lambda^{\xi'}_{(r_i'),I'}=d_{(r_i')}+n(\mu'+\nu)-n(\rho_\xi^{(t-1)}+\sigma_\xi^{(t-1)})+|\rho_\xi^{(t-1)}|.
\end{equation}
So if we can show that equality fails by 1 in the analogue of \eqref{eq:apply-lem-2.4}, i.e.\ that
\begin{equation} \label{eqn:failureby1}
d_{(r_i')}+n(\mu'+\nu)-n(\rho_\xi^{(t-1)}+\sigma_\xi^{(t-1)})+|\rho_\xi^{(t-1)}|=d_{(r_i')}+|\mu'|-1,
\end{equation}
then we will know by the induction hypothesis (or, in the case $s=t-1$, by Theorem~\ref{thm:kpstrata}(3)) that the conditions in Lemma~\ref{lem:signed-diagram-interpretation}(2) hold for all $j\leq t-2$. Alternatively, if we can show that equality holds in the analogue of \eqref{eq:apply-lem-2.4}, i.e.\ that
\begin{equation} \label{eqn:holding}
d_{(r_i')}+n(\mu'+\nu)-n(\rho_\xi^{(t-1)}+\sigma_\xi^{(t-1)})+|\rho_\xi^{(t-1)}|=d_{(r_i')}+|\mu'|,
\end{equation}
then we will know by the above proof of Conjecture~\ref{conj:strata}(2) (or, in the case $s=t-1$, by Theorem~\ref{thm:kpstrata}(2)) that the conditions in Lemma~\ref{lem:signed-diagram-interpretation}(1) hold for all $j\leq t-2$; clearly these are even stronger than those in Lemma~\ref{lem:signed-diagram-interpretation}(2). Assuming either of these eventualities, if we can also show that the condition in Lemma~\ref{lem:signed-diagram-interpretation}(2) holds for $j=t-1$, then we will have completed the induction step.

Now by Remark~\ref{rmk:lem-inequality} and the assumption on column-lengths of $\mu$ and $\nu$, saying that equality fails by exactly 1 in \eqref{eq:apply-lem-2.4} is equivalent to saying that $(\rho_\xi^{(t)};\sigma_\xi^{(t)})$ is obtained from $(\mu;\nu)$ by one of the following moves:
\begin{itemize}
\item a type-(1) move in which the box moves down a single row;
\item a type-(2) move in which the box moves down a single row; or
\item a type-(3) move in which a single box in row $\ell(\nu)+1$ moves from the bottom of a column of $\mu$ (necessarily of minimal length) to the bottom of a column of $\nu$ (necessarily of maximal length),
\end{itemize}
possibly followed, in the type-(3) case, by at most two type-(4) moves which have now become possible (because the type-(3) move has disrupted the property that all columns of $\mu$ are strictly longer than all columns of $\nu$). We now have various cases to consider. Recall that $\ell(\mu)=r_{t-1}$, so $\ell(\rho_\xi^{(t)}+\sigma_\xi^{(t)})$ is either $r_{t-1}$ or $r_{t-1}+1$.

\textbf{Case 1:} $\ell(\rho_\xi^{(t)})=\ell(\rho_\xi^{(t)}+\sigma_\xi^{(t)})=r_{t-1}$. This implies that the number of rows of $\xi_{t-1}$ containing a `$+$' box equals $r_{t-1}$, which as above forces $\xi_{t-1}$ to be the signed quasibipartition obtained by labelling every box of the diagram of $(\rho_\xi^{(t)};\sigma_\xi^{(t)})$ as `$+$', and then inserting `$-$' boxes between adjacent `$+$' boxes. Since every row ends with a `$+$', the condition in Lemma~\ref{lem:signed-diagram-interpretation}(2) holds for $j=t-1$. Moreover, we see that $\sigma_\xi^{(t-1)}=\sigma_\xi^{(t)}$ and $\rho_\xi^{(t-1)}$ is obtained from $\rho_\xi^{(t)}$ by deleting the longest column. Hence $(\rho_\xi^{(t-1)}; \sigma_\xi^{(t-1)})$ is obtained from $(\mu';\nu)$ by the `same' move (or sequence of moves) which produced $(\rho_\xi^{(t)};\sigma_\xi^{(t)})$ from $(\mu;\nu)$ (for which the possibilities were described above). This implies \eqref{eqn:failureby1}, finishing this case. See Figure \ref{fig:simul-degen} for an example (where $\xi_{t-1}$ is the bottom signed quasibipartition).
\begin{figure}[ht]
\[
\xymatrix{
&
\begin{tableau}
\row{\cp\cm\cp\cm\rce{+}\cm\cp\cm\cp}
\row{\cp\cm\cp\cm\rce{+}}
\row{\cp\cm\cp\cm\rce{+}}
\row{\q\q\q\q\rce{+}}
\end{tableau}
\ar@/_/[dl]_{+} \ar@/^/[dr]^{-}
& \\
\begin{tableau}
\caprow{\q\q(\mu;\nu)}
\caprow{\q\q\q\q\q\q}
\row{\c\c\rc\c\c}
\row{\c\c\rc}
\row{\bce{\bullet}\c\rc}
\row{\q\q\rc}
\end{tableau}
\ar[d]_{\text{type (1)}}
& &
\begin{tableau}
\caprow{\sq(\mu';\nu)}
\caprow{\q}
\row{\c\rc\c\c}
\row{\c\rc}
\row{\bce{\bullet}\rc}
\end{tableau}
\ar[d]^{\text{type (1)}}
\\
\begin{tableau}
\row{\c\c\rc\c\c}
\row{\c\c\rc}
\row{\q\c\rc}
\row{\q\bce{\bullet}\rc}
\caprow{\q\q\q\q\q\q}
\caprow{(\rho_\xi^{(t)}; \sigma_\xi^{(t)})}
\end{tableau}
& &
\begin{tableau}
\row{\q\q\c\rc\c\c}
\row{\q\q\c\rc}
\row{\q\q\q\rc}
\row{\q\q\q\brce{\bullet}}
\caprow{\q\q\q\q\q}
\caprow{(\rho_\xi^{(t-1)}; \sigma_\xi^{(t-1)})}
\end{tableau}
\\
&
\begin{tableau}
\row{\cp\cm\cp\cm\rce{+}\cm\cp\cm\cp}
\row{\cp\cm\cp\cm\rce{+}}
\row{\q \q \cp\cm\rce{+}}
\row{\q \q \cp\cm\rce{+}}
\row{\q}
\end{tableau}
\ar@/^/[ul]^{+} \ar@/_/[ur]_{-}
& \\
} 
\]
\caption{Simultaneous degeneration of subordinate biparitions}
\label{fig:simul-degen}
\end{figure}
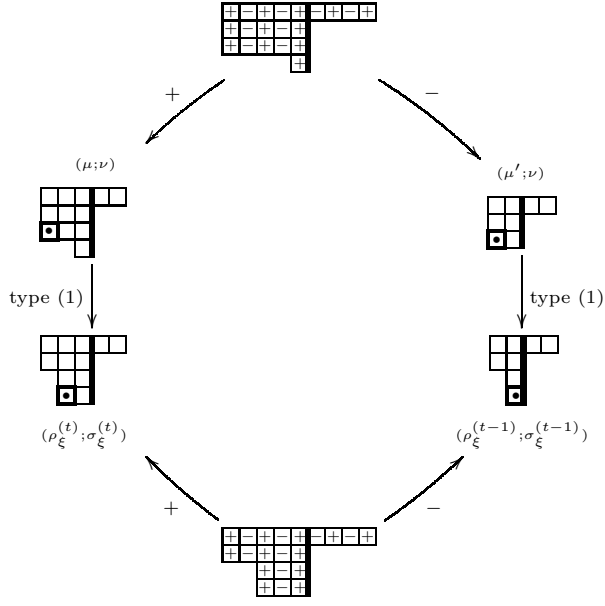

\textbf{Case 2:} $\ell(\rho_\xi^{(t)})<\ell(\rho_\xi^{(t)}+\sigma_\xi^{(t)})=r_{t-1}$. This case occurs if and only if $\mu$ has a single column (so $s=t-1$), $\ell(\nu)=r_{t-1}-1$, and to form $(\rho_\xi^{(t)};\sigma_\xi^{(t)})$ we make a type-(3) move. It is still true that the number of rows of $\xi_{t-1}$ containing a `$+$' box equals $r_{t-1}$, so as in Case 1, $\xi_{t-1}$ is the signed quasibipartition obtained by labelling every box of the diagram of $(\rho_\xi^{(t)};\sigma_\xi^{(t)})$ as `$+$', and then inserting `$-$' boxes between adjacent `$+$' boxes. Once again, every row ends with a `$+$'. Since $\rho_\xi^{(t-1)}$ is empty and $\sigma_\xi^{(t-1)}=\nu$, \eqref{eqn:holding} holds, finishing this case.

\textbf{Case 3:} $\ell(\rho_\xi^{(t)}+\sigma_\xi^{(t)})=r_{t-1}+1$ and $|\rho_\xi^{(t)}|=|\mu|$. This case occurs when $\mu$ has more than one column of length $r_{t-1}$ (i.e.\ $\mu_{r_{t-1}}\geq 2$), and to form $(\rho_\xi^{(t)};\sigma_\xi^{(t)})$ we move the corner box in row $r_{t-1}$ down to row $r_{t-1}+1$ (this is either a type-(1) move, or, if it happens that $\ell(\nu)=r_{t-1}-1$, a type-(3) move followed by a type-(4) move). That is, $\sigma_\xi^{(t)}=\nu$ and $\rho_\xi^{(t)}=\widetilde{\mu}$, where $\widetilde{\mu}$ is the partition $(\mu_1,\mu_2,\cdots,\mu_{r_{t-1}-1},\mu_{r_{t-1}}-1,1)$. In this case the number of rows of $\xi_{t-1}$ containing a `$+$' box equals $r_{t-1}+1$, so there is apparently more freedom in the choice of $\xi_{t-1}$: after labelling every box of $(\widetilde{\mu};\nu)$ as `$+$', and then inserting `$-$' boxes between adjacent `$+$' boxes, we still have an additional `$-$' box to place.

The possibilities are constrained, however, by the analogue of \eqref{eqn:dominance}, which ensures that $(\rho_\xi^{(t-1)}; \sigma_\xi^{(t-1)}) \le (\mu';\nu)$. The forced `$-$' boxes in $\xi_{t-1}$ have the shape of the bipartition $(\widetilde{\mu}';\nu)$, where $\widetilde{\mu}'$ is obtained from $\mu'$ by deleting the corner box in row $r_{t-1}$. Hence $(\rho_\xi^{(t-1)}; \sigma_\xi^{(t-1)})$ either equals $(\mu';\nu)$, or is obtained from $(\mu';\nu)$ by moving this same corner box to row $r_{t-1}$ on the $\nu$ side (a type-(3) move) or to row $r_{t-1}+1$ on the $\mu'$ side (a type-(1) move, or a type-(3) move followed by a type-(4) move). If $(\rho_\xi^{(t-1)}; \sigma_\xi^{(t-1)})=(\mu';\nu)$, then \eqref{eqn:holding} holds. If a move is required, then \eqref{eqn:failureby1} holds. What remains, for this case, is to verify the condition in Lemma~\ref{lem:signed-diagram-interpretation}(2) for $j=t-1$.  

Suppose that $\ell(\rho_\xi^{(t-1)}+\sigma_\xi^{(t-1)})=r_{t-1}$, which is equivalent to saying that the additional `$-$' box in $\xi_{t-1}$ is in row $r_{t-1}$. Then row $r_{t-1}+1$ of $\xi_{t-1}$ consists of a single `$+$' box immediately left of the wall; moreover, it is either true that every row of $\xi_{t-1}$ begins with a `$+$' or that every row of $\xi_{t-1}$ ends with a `$+$', since the additional `$-$' box cannot falsify both statements. So in this event, we are finished. 

The other possibility is that $(\rho_\xi^{(t-1)}; \sigma_\xi^{(t-1)})$ is obtained from $(\mu';\nu)$ by moving the corner box in row $r_{t-1}$ on the $\mu'$ side to row $r_{t-1}+1$ on the $\mu'$ side. Here we observe that since equality holds in \eqref{eq:apply-prop-6.1}, there cannot be any rearrangement of parts in forming the subordinate bipartition $(\rho_\xi^{(t-1)}; \sigma_\xi^{(t-1)})$ from $\xi_{t-1}$ (see the comment following \eqref{eqn:dim-np}). So it is not possible that the additional `$-$' box in $\xi_{t-1}$ is in a row by itself (immediately right of the wall), following row $r_{t-1}+1$, which consists of a single `$+$' box (immediately left of the wall). Hence the additional `$-$' box must be in row $r_{t-1}+1$, either before or after the single `$+$' box. Again, it follows that either every row of $\xi_{t-1}$ begins with a `$+$' or every row of $\xi_{t-1}$ ends with a `$+$'. If the additional `$-$' box comes before the `$+$' box in row $r_{t-1}+1$, then that row ends with a box immediately left of the wall. If the additional `$-$' box comes after the `$+$' box, and therefore right of the wall, then the fact that it is brought to the left of the wall in forming $(\rho_\xi^{(t-1)}; \sigma_\xi^{(t-1)})$ implies that there is a row of $\xi_{t-1}$ which contains no `$-$' boxes right of the wall, and therefore ends with a box immediately left of the wall. So the condition in Lemma~\ref{lem:signed-diagram-interpretation}(2) holds.

\textbf{Case 4:} $\ell(\rho_\xi^{(t)}+\sigma_\xi^{(t)})=r_{t-1}+1$ and $|\rho_\xi^{(t)}|\neq|\mu|$. This case occurs when all columns of $\mu$ have length $r_{t-1}$, $\ell(\nu)=r_{t-1}-1$, and to form $(\rho_\xi^{(t)};\sigma_\xi^{(t)})$ from $(\mu;\nu)$ we move the corner box in row $r_{t-1}-1$ on the $\nu$ side down to row $r_{t-1}+1$ on the $\mu$ side (this is a type-(3) move followed by two type-(4) moves, or a type-(3) move followed by a single type-(4) move of two boxes). As in the previous case, after labelling every box of $(\rho_\xi^{(t)};\sigma_\xi^{(t)})$ as `$+$', and then inserting `$-$' boxes between adjacent `$+$' boxes, we have to place one additional `$-$' box to form $\xi_{t-1}$. Again, the possibilities are constrained by the fact that $(\rho_\xi^{(t-1)}; \sigma_\xi^{(t-1)}) \le (\mu';\nu)$. The forced `$-$' boxes in $\xi_{t-1}$ have the shape of the bipartition $(\mu';\widetilde{\nu})$, where $\widetilde{\nu}$ is obtained from $\nu$ by deleting the corner box in row $r_{t-1}-1$. Hence $(\rho_\xi^{(t-1)}; \sigma_\xi^{(t-1)})$ either equals $(\mu';\nu)$, or is obtained from $(\mu';\nu)$ by moving this same corner box to row $r_{t-1}$ on the $\nu$ side or to row $r_{t-1}+1$ on the $\mu'$ side. If $(\rho_\xi^{(t-1)}; \sigma_\xi^{(t-1)})=(\mu';\nu)$, then \eqref{eqn:holding} holds. If a move is required, then \eqref{eqn:failureby1} holds. The verification that the condition in Lemma~\ref{lem:signed-diagram-interpretation}(2) holds for $j=t-1$ is almost identical to the previous case.
\end{proof}

%

To conclude this section, here is a list of the cases of Conjecture~\ref{conj:enhnorm} whose proof is now complete.

\begin{cor} \label{cor:main}
The enhanced nilpotent orbit closure $\overline{\cO_{\mu;\nu}}$ is normal in the following cases:
\begin{enumerate}
\item $n \le 6$;
\item $\mu_1^\bt\leq\nu_{\nu_1}^\bt$ \textup{(}i.e., every column of the diagram of $\nu$ is at least as long as every column of $\mu$\textup{)};
\item $\mu_{\mu_1}^\bt>\nu_{1}^\bt$ \textup{(}i.e., every column of $\mu$ is longer than every column of $\nu$\textup{)}.
\end{enumerate}
\end{cor}
\begin{proof}
These are the bipartitions corresponding to the enhanced quiver varieties which are proved to be normal by, respectively, our computer verifications of Conjecture~\ref{conj:strata}, Theorem~\ref{thm:easy}, and Theorem~\ref{thm:main}.
\end{proof}
\section{Regularity in Codimension~$1$}
\label{sect:r1}

In this last section, we prove that every enhanced nilpotent orbit closure is regular in codimension~$1$.  Of course, this would be an immediate consequence of Conjecture~\ref{conj:enhnorm}. However, the results in this section hold over any algebraically closed field $\F$, unlike the proposed method of proof of Conjecture~\ref{conj:enhnorm}, which requires $\F$ to be $\C$. 


We first prove the smoothness of certain unions of orbits in the enhanced nilpotent cone $V\times\cN$. As before, $n$ denotes the dimension of $V$.
If $\lambda$ is a partition of $n$, we write $U_\lambda$ for $V\times\cO_\lambda$, which is a locally closed subvariety of $V\times\cN$. Let $\cQ_\lambda$ be the subset of $\cQ_n$ consisting of bipartitions $(\mu;\nu)$ such that $\mu+\nu=\lambda$.

\begin{prop} \label{prop:union1} 
Let $\lambda$ be any partition of $n$.
\begin{enumerate}
\item $U_\lambda$ is the union of the orbits $\cO_{\mu;\nu}$ for $(\mu;\nu)\in\cQ_\lambda$.
\item For $(\rho;\sigma),(\mu;\nu)\in\cQ_\lambda$, $\cO_{\rho;\sigma}\subset\overline{\cO_{\mu;\nu}}$ if and only if $\rho_i\leq\mu_i$ for all $i$.
\item For $(\mu;\nu)\in\cQ_\lambda$,
\[ \overline{\cO_{\mu;\nu}}\cap U_\lambda=\{(v,x)\in U_\lambda\,|\,x^{\mu_i}v\in\im(x^{\lambda_i}),\text{ for all }i\}. \]
\item In $U_\lambda$, the closure of each orbit is smooth; that is, for every $(\mu;\nu)\in\cQ_\lambda$, $\overline{\cO_{\mu;\nu}}\cap U_\lambda$ is smooth.
\end{enumerate}
\end{prop}

\begin{proof}
Part (1) is equivalent to \eqref{eqn:bipartitionsum}. Part (2) follows from Lemma~\ref{lem:enhcl-crit}. Using the normal basis given in Lemma~\ref{lem:enh-param}, one sees that if $(v,x)\in\cO_{\mu;\nu}\subset U_\lambda$, then for all $i$,
\[ \min\{s\,|\,x^s v\in\im(x^{\lambda_i})\}=\mu_i. \]
Combining this with part (2), we deduce part (3). From this part (4) follows, because the projection $(v,x)\mapsto x$ exhibits $\overline{\cO_{\mu;\nu}}\cap U_\lambda$ as a vector bundle over $\cO_\lambda$, where the fibre over $x$ is the vector subspace $\bigcap_i (x^{\mu_i})^{-1}(\im(x^{\lambda_i}))$ of $V$. (Incidentally, this subspace can alternatively be described as $\sum_i x^{\nu_i}(\ker(x^{\lambda_i}))$.)
\end{proof}

If $0\leq m\leq n$ and $\pi$ is a partition of $n-m$, define 
\[ U_{m,\pi}=\{(v,x)\in V\times\cN\,|\,\dim\F[x]v=m,\, x|_{V/\F[x]v}\in\cO_\pi\}, \]
which is a locally closed subvariety of $V\times\cN$. Here $\F[x]v$ is the span of the elements $x^i v$ for all $i$, which is obviously an $x$-stable subspace of $V$; since $x$ is nilpotent, to say that $\dim\F[x]v=m$ is to say that $m$ is minimal such that $x^m v=0$. 

Let $\cQ_{m,\pi}$ be the subset of $\cQ_n$ consisting of bipartitions $(\mu;\nu)$ such that $\mu_1=m$ and $\mu[1]+\nu=\pi$, where $\mu[1]$ denotes the partition $(\mu_2,\mu_3,\cdots)$. The map $(\mu;\nu)\mapsto\mu+\nu$ gives a bijection 
\[ \cQ_{m,\pi}\longleftrightarrow\{\lambda\in\cP_n\,|\,\lambda_1\geq\pi_1\geq\lambda_2\geq\pi_2\geq\cdots\}. \]

\begin{prop} \label{prop:union2}
Let $m$ and $\pi$ be as above. 
\begin{enumerate}
\item $U_{m,\pi}$ is the union of the orbits $\cO_{\mu;\nu}$ for $(\mu;\nu)\in\cQ_{m,\pi}$.
\item For $(\rho;\sigma),(\mu;\nu)\in\cQ_{m,\pi}$, $\cO_{\rho;\sigma}\subset\overline{\cO_{\mu;\nu}}$ if and only if $\sigma_i\leq\nu_i$ for all $i$, which in turn happens if and only if $\cO_{\rho+\sigma}\subset\overline{\cO_{\mu+\nu}}$.
\item For $(\mu;\nu)\in\cQ_{m,\pi}$,
\[
\overline{\cO_{\mu;\nu}}\cap U_{m,\pi}=\{(v,x)\in U_{m,\pi}\,|\,x^{m+\nu_i}((x^{\pi_i})^{-1}(\F[x]v))=0,\text{ for all }i\}.
\]
\item In $U_{m,\pi}$, the closure of each orbit is smooth; that is, for every $(\mu;\nu)\in\cQ_{m,\pi}$, $\overline{\cO_{\mu;\nu}}\cap U_{m,\pi}$ is smooth.
\end{enumerate}
\end{prop}

\begin{proof}
Part (1) is proved in \cite[Lemma 2.5]{ah}, and part (2) follows from Lemma~\ref{lem:enhcl-crit}. Using the basis defined in \cite[Lemma 2.5]{ah}, one sees that if $(v,x)\in\cO_{\mu;\nu}\subset U_{m,\pi}$, then for all $i$,
\[
\min\{s\,|\,x^s((x^{\pi_i})^{-1}(\F[x]v))=0\}=m+\nu_i.
\] 
(To verify that the minimum is at least $m+\nu_i$, note that the basis element $w_{i,\mu_i+\nu_i}$ belongs to $(x^{\pi_i})^{-1}(\F[x]v)$, and that $x^{m+\nu_i-1}(w_{i,\mu_i+\nu_i})\neq 0$.) Combining this with part (2), we deduce part (3).

It remains to prove part (4). Let $Z_{m,\pi}$ be the variety of triples $(W,y,z)$ where $W$ is an $m$-dimensional subspace of $V$, $y$ is a nilpotent endomorphism of $W$ with a single Jordan block, and $z$ is a nilpotent endomorphism of $V/W$ which belongs to $\cO_{\pi}$. It is clear that $Z_{m,\pi}$ is a homogenous variety for $GL(V)$. We have a $GL(V)$-equivariant fibre bundle 
\[ U_{m,\pi}\to Z_{m,\pi}:(v,x)\mapsto(\F[x]v,x|_{\F[x]v},x|_{V/\F[x]v}), \] 
in which the fibre over $(W,y,z)$ is $(W\setminus\ker(y^{m-1}))\times A_{W,y,z}$, where 
\[ A_{W,y,z}=\{x\in\gl(V)^W\,|\,x|_W=y,\,x|_{V/W}=z\}, \] 
an affine-linear subspace of $\gl(V)^W$ (the parabolic subalgebra of $\gl(V)$ stabilizing $W$). So it suffices to show that $\overline{A_{\mu;\nu}}=\overline{\cO_{\mu+\nu}}\cap A_{W,y,z}$ is smooth. By part (3),
\[
\overline{A_{\mu;\nu}}=\{x\in A_{W,y,z}\,|\,x^{m+\nu_i}(W+\ker(z^{\pi_i}))=0,\text{ for all }i\},
\]
where $W+\ker(z^{\pi_i})$ denotes the preimage of $\ker(z^{\pi_i})$ under the projection $V\to V/W$. Fixing a base-point $x_0\in A_{W,y,z}$, one has $A_{W,y,z}=x_0+\fn^W$ where $\fn^W$ is the nilpotent radical of $\gl(V)^W$. For any $k$, the matrix coefficients of $x^k-x_0^k$ are linear functions of $x-x_0\in\fn^W$, so the condition $x^{m+\nu_i}(W+\ker(z^{\pi_i}))=0$ translates into a linear condition on $x-x_0$. Hence $\overline{A_{\mu;\nu}}$ is an affine-linear subspace of $\gl(V)^W$, and is smooth as required. 
\end{proof}

\begin{thm}\label{thm:r1}
For every $(\mu;\nu)\in\cQ_n$, $\overline{\cO_{\mu;\nu}}$ is regular in codimension~$1$.
\end{thm}

\begin{proof}
Let $\lambda=\mu+\nu$, $m=\mu_1$, and $\pi=\mu[1]+\nu$. Suppose that $\cO_{\rho;\sigma}$ has codimension $1$ in $\overline{\cO_{\mu;\nu}}$. From the description of covering relations given in Section~\ref{sect:covering} and the dimension formula \eqref{eqn:dim-norb}, it follows that either $(\rho;\sigma)\in\cQ_\lambda$ (in the case of a type (3) move of a single box) or $(\rho;\sigma)\in\cQ_{m,\pi}$ (in the case of a type (4) move of a single box). So $\cO_{\rho;\sigma}$ is contained in either $\overline{\cO_{\mu;\nu}}\cap U_\lambda$ or $\overline{\cO_{\mu;\nu}}\cap U_{m,\pi}$, both of which are open in $\overline{\cO_{\mu;\nu}}$ and smooth by Propositions \ref{prop:union1} and \ref{prop:union2}. So $\overline{\cO_{\mu;\nu}}$ is smooth at all points of $\cO_{\rho;\sigma}$, proving the result.
\end{proof}


\end{document}